%% file: enr_alg_th.tex
\documentclass[12pt]{article}

\input{pre}

\begin{document}

\author{\normalsize  Rory B. B. Lucyshyn-Wright\thanks{The author gratefully acknowledges financial support in the form of an AARMS Postdoctoral Fellowship, a Mount Allison University  Research Stipend, and, earlier, an NSERC Postdoctoral Fellowship.}\let\thefootnote\relax
\\
\small Mount Allison University, Sackville, New Brunswick, Canada}

\title{\large \textbf{Enriched algebraic theories and monads for a system of arities}}

\date{}

\maketitle

\abstract{
Under a minimum of assumptions, we develop in generality the basic theory of universal algebra in a symmetric monoidal closed category $\V$ with respect to a specified system of arities $j:\J \hookrightarrow \V$.  Lawvere's notion of \textit{algebraic theory} generalizes to this context, resulting in the notion of \textit{single-sorted $\V$-enriched $\J$-cotensor theory}, or $\J$\nolinebreak\mbox{-}\nolinebreak\textit{theory} for short.  For suitable choices of $\V$ and $\J$, such $\J$-theories include the enriched algebraic theories of Borceux and Day, the enriched Lawvere theories of Power, the equational theories of Linton's 1965 work, and the $\V$-theories of Dubuc, which are recovered by taking $\J = \V$ and correspond to arbitrary $\V$-monads on $\V$.  We identify a modest condition on $j$ that entails that the $\V$\nolinebreak\mbox{-}\nolinebreak category of $\T$-algebras exists and is monadic over $\V$ for every $\J$-theory $\T$, even when $\T$ is not small and $\V$ is neither complete nor cocomplete.  We show that $j$ satisfies this condition if and only if $j$ presents $\V$ as a free cocompletion of $\J$ with respect to the weights for left Kan extensions along $j$, and so we call such systems of arities \textit{eleutheric}.  We show that $\J$-theories for an eleutheric system may be equivalently described as (i) monads in a certain one-object bicategory of profunctors on $\J$, and (ii) $\V$-monads on $\V$ satisfying a certain condition.  We prove a characterization theorem for the categories of algebras of $\J$-theories, considered as $\V$-categories $\A$ equipped with a specified $\V$-functor $\A \rightarrow \V$.
}

\section{Introduction} \label{sec:intro}

In the 1930s, Birkhoff laid the foundations of the subject of \textit{universal algebra} \cite[II]{Bir:SelPa}, which provides general methods for the study of algebraic objects described by operations and equations, such as groups, rings, lattices, Boolean algebras, and so on.  In 1963, Lawvere \cite{Law:PhD} provided an elegant formulation of universal algebra through category theory, wherein an \textit{algebraic theory} is\footnote{This way of stating the definition appears in \cite{Law:ProbsAlgTh} and \cite[Vol. 2]{Bor}, for example.} simply a category $\T$ having a denumerable set of objects $S^0,S^1,S^2,...$ such that $S^n$ is an $n$-th power of an object $S = S^1$ of $\T$ called the \textit{sort}.  A \textit{$\T$-algebra} is then defined as a functor $A:\T \rightarrow \Set$ that preserves finite powers\footnote{Typically one demands preservation of finite \textit{products}, but this amounts to the same.}.  We construe $\ca{A} = A(S)$ as the \textit{underlying set} or \textit{carrier} of $A$, and therefore $A(S^n)$ is simply an $n$-th power $\ca{A}^n$ of the carrier set.  Morphisms $\omega:S^n \rightarrow S$ in $\T$ may be called \textit{abstract $n$-ary operations}, and the associated mappings $A(\omega):\ca{A}^n \rightarrow \ca{A}$ are called \textit{(concrete) $n$-ary operations}.  We call $n$ the \textit{arity} of the operation.  Linton \cite{Lin:Eq} varied this formulation to allow infinite arities, so that the role of the finite cardinals $n$ in the above is played instead by arbitrary sets $J$ and the objects of a theory $\T$ are $J$-th powers $S^J$ of a single object $S$.  Linton found that those equational theories $\T$ that are locally small correspond to \textit{monads} $\TT = (T,\eta,\mu)$ on $\Set$, the key idea being that $TJ \cong \T(S^J,S)$.  Lawvere's finitary algebraic theories then correspond to those monads $\TT$ on $\Set$ that are \textit{finitary}, meaning that $T$ preserves filtered colimits.

Several variations and generalizations on Lawvere's notion of algebraic theory begin with the idea of replacing the category of sets with a given symmetric monoidal closed category $\V$ and taking $\T$ to be a $\V$-enriched category rather than an ordinary category.  In the formulation of Borceux and Day \cite{BoDay}, an algebraic theory $\T$ is a $\V$-category whose objects are (conical) finite powers $S^n$ of a single object $S = S^1$, so that one has for each finite cardinal $n$ an \textit{object of abstract $n$-ary operations} $\T(S^n,S)$ in $\V$.  A $\T$-algebra is then a finite-power-preserving $\V$-functor $A:\T \rightarrow \V$, so that the passage from abstract to concrete operations is implemented by a family of morphisms in $\V$.  More drastically, one can also take the arities themselves to be objects $J$ of $\V$.  For example, building on work of Kelly \cite{Ke:FL}, Power \cite{Pow:EnrLaw} takes the arities $J$ to be the \textit{finitely presentable} objects of $\V$, assuming that $\V$ is locally finitely presentable as a closed category.  Power's notion of \textit{enriched Lawvere theory} is therefore a $\V$-category $\T$ whose objects $S^J$ are $J$-th \textit{cotensors}\footnote{In the body of the paper, we write cotensors $C^J$ in a $\V$-category $\C$ instead as $[J,C]$.} of a single object $S = S^I$, where $I$ is the unit object of $\V$, the notion of cotensor here providing the appropriate concept of `$\V$-enriched $J$-th power'.  Consequently a $\T$-algebra $A:\T \rightarrow \V$ is defined as a $\V$-functor that preserves $J$-cotensors for finitely presentable objects $J$ of $\V$, so that $A(S^J) \cong \uV(J,\ca{A})$ is the internal hom from $J$ to $\ca{A}$ in $\V$.

More basically, one can also form an analogue of Linton's notion of equational theory for an arbitrary symmetric monoidal closed category $\V$ by taking as arities \textit{arbitrary} objects $J$ of $\V$.  Such enriched theories with arbitrary arities were introduced by Dubuc \cite{Dub:EnrStrSem} in 1970, under the name \textit{$\V$-theories},  and they are equivalently described as arbitrary $\V$-monads $\TT$ on $\V$ \pbref{thm:vth_vmnd}, as Dubuc showed under the blanket assumption that $\V$ is complete and well-powered.

In the present paper, we put forth a very simple and general notion of enriched algebraic theory that encompasses all of the above examples and more, and we develop the basic theory of such enriched algebraic theories under a minimum of assumptions.  In particular we do not assume that $\V$ is complete or cocomplete.  We begin with a given \textit{system of arities}, by which we understand a fully faithful strong symmetric monoidal $\V$-functor $j:\J \rightarrowtail \V$.  Without loss of generality, we can take $j:\J \hookrightarrow \V$ to be the inclusion of a full subcategory containing $I$ and closed under $\otimes$, and we construe the objects $J$ of $\J$ as allowable arities.  A $\J$-\textit{theory} is then a $\V$-category $\T$ whose objects are cotensors $S^J$ of an object $S = S^I$ with $J \in \J$, and in fact we require (without loss of generality) that $\ob\T = \ob\J$.  $\T$-algebras are then defined as $\J$-cotensor-preserving $\V$-functors $A:\T \rightarrow \V$.  Our $\J$-theories may be identified quite unambiguously among various possible notions of algebraic theory as the \textit{single-sorted $\V$-enriched $\J$-cotensor theories}.  Theories of the Borceux-Day type, which we may call \textit{single-sorted enriched finite power theories}, are recovered by considering as arities the finite copowers $n \cdot I$ of the unit object $I$.  As a more unusual example, if we instead take $\J$ to be the one-object full subcategory $\{I\} \hookrightarrow \V$ with the trivial symmetric monoidal structure on $\{I\}$, then $\{I\}$-theories $\T$ are the same as monoids $R$ in $\V$, and $\T$-algebras are the same as left $R$-modules in $\V$.

On this basis, we investigate modest conditions on a $\J$-theory $\T$ that entail the existence of the $\V$-category $\Alg{\T}$ of $\T$-algebras and its monadicity over $\V$ \pbref{thm:existence_and_vmonadicity_of_talg}, even when $\T$ is not small and $\V$ is neither complete nor cocomplete.  We study a correspondingly modest condition on $j:\J \hookrightarrow \V$ that entails these conclusions for every $\J$-theory $\T$, calling a system of arities \textit{eleutheric} if it satisfies this condition \pbref{def:eleutheric}.  The choice of name was prompted by our theorem \bref{thm:eleutheric} to the effect that a system of arities $j:\J \hookrightarrow \V$ is eleutheric if and only if $j$ presents $\V$ as a free $\Phi$-cocompletion of $\J$ for the class $\Phi$ of weights for left Kan extensions along $j$, whereas in Greek, \textit{eleutheros} ($\varepsilon\lambda\varepsilon\acute{\upsilon}\theta\varepsilon\rho o\varsigma$) means \textit{free}.

We show that for an eleutheric system of arities, $\J$-theories are the same as monads in a one-object bicategory whose 1-cells are certain $\V$-profunctors on $\J$ \pbref{thm:th_as_mnd_crprofj}.  We then show that this bicategory is equivalent to a one-object 2-category consisting of certain endo-$\V$-functors on $\V$ \pbref{thm:crprofj_bieq_jary_end}.  On this basis, we establish an equivalence between $\J$-theories and $\J$-\textit{ary $\V$-monads} on $\V$ \pbref{thm:jth_jary_mnd}, which we define as $\V$-monads that preserve left Kan extensions along $j$.  In this setting, we then prove a characterization theorem for $\J$-\textit{algebraic} $\V$-categories over $\V$ \pbref{thm:charn_jalg_cats_over_v}, i.e. we characterize among arbitrary $\V$-functors $\A \rightarrow \V$ those equivalent to the forgetful $\V$-functor $\Alg{\T} \rightarrow \V$ for some $\J$-theory $\T$.

Our work should be compared to Lack and Rosick\'y's succinct development of \textit{Lawvere $\Phi$-theories} \cite[\S 7]{LaRo} for a class of weights $\Phi$, which can be understood as follows.  Prior to Power's work on enriched Lawvere theories, Kelly had studied $\V$-enriched \textit{finite limit theories} \cite{Ke:FL} for $\V$ locally finitely presentable as a closed category, an instance of Kelly's more general $\Phi$-limit theories for a class of weights $\Phi$ \cite[Ch. 6]{Ke:Ba}.  Lack and Rosick\'y masterfully demonstrate that the relation between Power's enriched Lawvere theories and Kelly's finite limit theories generalizes to the setting of a given class of weights $\Phi$ satisfying certain axioms, with $\V$ complete and cocomplete.  In particular, they obtain a notion of \textit{Lawvere $\Phi$-theory} in which the role of the arities $J$ is played by the \textit{$\Phi$-presentable} objects of a locally $\Phi$-presentable $\V$-category $\K$.

The Lawvere $\Phi$-theories of Lack and Rosick\'y provide a certain generality yet do not capture all of the above examples, and the brief treatment of them in \cite[\S 7]{LaRo} proceeds under stronger assumptions than we make here, including their Axiom A as well as the completeness and cocompleteness of $\V$ and the assumption that $\Phi$ is a \textit{locally small} class of weights.  Also, assertions made in \cite[\S 7]{LaRo} entail immediately that their Lawvere $\Phi$-theories are necessarily (essentially) small\footnote{However, a proof to this effect is not given therein, and it does not seem obvious why this should be the case, despite the assumption of local smallness of $\Phi$ therein.}, and this property is used therein.  Hence the Lack-Rosick\'y framework of Lawvere $\Phi$-theories does not include\footnote{It would be natural to employ here a class of weights $\Phi$ containing the weights for all cotensors, but such a class $\Phi$ would not in general be locally small.} the important example of $\J$-theories with $\J = \V$, equivalently, $\V$-monads on $\V$.  Also, it is not clear that the enriched theories of Borceux and Day can be captured as Lawvere $\Phi$-theories in the given sense, despite the fact that Lack and Rosick\'y discuss the Borceux-Day work and the specific class $\Phi$ of weights for conical finite products \cite[\S 5.2]{LaRo}.  Indeed, the finite copowers $n \cdot I$ that serve as the arities for the Borceux-Day theories are not in general the same as the $\Phi$-presentable objects for this class (\S \bref{sec:appendix}).

Hence, encouraged by the simplicity and generality of the notion of $\J$-theory, we felt a need to develop in generality the fundamentals of enriched universal algebra with respect to a given system of arities $\J \hookrightarrow \V$.  The most pressing practical reason for such a development is the need for a common framework for work on enriched universal-algebraic notions that may depend on a choice of arities, so that such work is then applicable at once to arbitrary $\V$-monads on $\V$, the enriched Lawvere theories of Power, the theories of Borceux and Day, and so on.  In particular, the author's forthcoming study of \textit{commutation} and \textit{commutants} for enriched algebraic theories has necessitated such a framework and, in its turn, forms the basis for a theory of \textit{measure and distribution monads} presented in a recent conference talk \cite{Lu:CT2015}.

In the context of ordinary (unenriched) category theory, algebraic theories with respect to a given subcategory of arities $j:\J \hookrightarrow \C$ in an ordinary category $\C$ (or even just a functor $j$) were considered as early as Linton's work in the proceedings of the Z\"urich seminar of 1966/67 \cite{Lin:OutlFuncSem}.  Further related work in unenriched category theory includes \cite{BMW} on monads and theories with arities and \cite{ACU} on $J$-relative monads.  Also, our result connecting $\J$-theories and free $\Phi$-cocompletion \pbref{thm:eleutheric} bears an interesting relation to the ideas of \cite{Hy:ThAlgTh} on a general framework for universal algebraic notions in terms of Kleisli bicategories.  Indeed, by passing to an enlargement $\V'$ of $\V$ \pbref{par:enl} one can (by \cite[3.6]{KeSch}) form the free $\Phi$-cocompletion pseudomonad on $\eCAT{\V'}$ for the class of weights $\Phi$ for left Kan extensions along a given eleutheric system $j:\J \hookrightarrow \V$.  From this perspective, Theorems \bref{thm:eleutheric} and \bref{thm:jth_jary_mnd} below provide ample reason to expect that $\J$-theories may be equivalently described as monads on $\J$ in the associated Kleisli bicategory, which in turn provide a suitable notion of $j$-relative $\V$-monad on $\V$.

\section{Background and notation}\label{sec:background}

\begin{ParSub}\label{par:enl}
Throughout, we fix a symmetric monoidal closed category $\V$, and we employ the theory of $\V$-enriched categories (\cite{EiKe,Dub,Ke:Ba}).  We will distinguish terminologically between \textit{$\V$-categories} and (ordinary) \textit{categories}, and similarly for functors, etcetera, but the concepts of limit, Kan extension, density and so on shall be interpreted in the enriched sense when applied to enriched data.  We denote by $\uV$ the $\V$-category canonically associated to $\V$, so that the underlying ordinary category of $\uV$ may be identified with $\V$ itself.  We do not assume the existence of any limits or colimits in $\V$.  Hence we cannot in general form the $\V$-functor $\V$-category $[\A,\B]$ for $\V$-categories $\A$ and $\B$, even when $\A$ is small.  We will however make some use of the method of passing to larger universe $\U$ with respect to which certain given $\V$-categories such as $\A$ are small (i.e., $\U$-small) and embedding $\V$ into a ($\U$-)complete and ($\U$-)cocomplete symmetric monoidal closed category $\V'$, per \cite[\S 3.11, 3.12]{Ke:Ba}, such that the embedding $\V \hookrightarrow \V'$ is strong symmetric monoidal and preserves all limits that exist in $\V$ and all $\U$-small colimits that exist in $\V$.  We shall call any such $\V'$ an \textit{enlargement} of $\V$.  As is common practice, e.g. in \cite{Ke:Ba}, we write as if the given embedding of $\V$ into $\V'$ is a strict monoidal inclusion.
\end{ParSub}

\begin{ParSub}\label{par:weighted_limits}
We shall make extensive use the notion of $\V$-enriched \textbf{weighted (co)limit}, called \textit{indexed limit} in \cite{Ke:Ba}.  A \textbf{weight} on a $\V$-category $\B$ is simply a $\V$-functor $W:\B^\op \rightarrow \uV$.  We call a pair $(W,D)$ consisting of $\V$-functors $W:\B^\op \rightarrow \uV$ and $D:\B \rightarrow \C$ a \textbf{weighted diagram} of shape $\B$ in the $\V$-category $\C$.  A \textit{(weighted) colimit} of $(W,D)$ is an object $W \star D$ of $\C$ equipped with a family of morphisms $\pi_B^C:\C(W \star D,C) \rightarrow \uV(WB,\C(DB,C))$ $(B \in \ob\B, C \in \ob\C)$ that exhibit $\C(W \star D,C)$ as an object of $\V$-natural transformations $[\B^\op,\uV](W,\C(D-,C))$ for each fixed object $C$ of $\C$.  By definition, a \textbf{cylinder} $(C,\gamma)$ on $(W,D)$ consists of an object $C$ of $\C$ together with a $\V$-natural transformation $\gamma:W \Rightarrow \C(D-,C)$.  A weighted colimit can be defined equivalently as a cylinder $(W \star D,\lambda)$ on $(W,D)$ satisfying a certain condition \cite[\S 3.1]{Ke:Ba}.  Dual to the notion of weighted colimit is the notion of \textit{(weighted) limit} $\{U,E\}$ where $U:\B \rightarrow \uV$ and $E:\B \rightarrow \C$ (\cite[\S 3.1]{Ke:Ba}).  In particular, we have the notion of \textbf{cotensor} $[V,C]$ of an object $C$ of $\C$ by an object $V$ of $\V$ \cite[\S 3.7]{Ke:Ba}.
\end{ParSub}

\begin{ParSub}\label{par:det_cond_pres_refl_colims}
Given a class of (possibly large) weighted diagrams $\Phi$ and a $\V$-category $\A$, let us write $\Phi_\A$ for the subclass of $\Phi$ consisting of weighted diagrams in $\A$.  Colimits $W \star D$ with $(W,D) \in \Phi$ are called \textit{$\Phi$-colimits}.  We say that $\A$ is $\Phi$-\textbf{cocomplete} if it has all $\Phi$-colimits.  A $\V$-functor $G:\A \rightarrow \C$ is said to be $\Phi$-\textbf{cocontinuous} if it preserves $\Phi$-colimits.  We say that $G$ \textbf{detects} $\Phi$-colimits if for any $(W,D) \in \Phi_\A$, if $W \star (GD)$ exists then $W \star D$ exists.  $G$ \textbf{conditionally preserves} $\Phi$-colimits if for each $(W,D) \in \Phi_\A$ with a colimit $W \star D$, if $W \star (GD)$ exists then $G$ preserves the colimit $W \star D$.  $G$ \textbf{reflects} $\Phi$-colimits if for any cylinder $(A,\gamma)$ on $(W,D) \in \Phi_\A$, if the associated cylinder $(GA,G\gamma)$ on $(W,GD)$ is a colimit cylinder, then $(A,\gamma)$ is a colimit cylinder on $(W,D)$.  $G$ \textbf{creates} $\Phi$-colimits if for any $(W,D) \in \Phi_\A$ and any colimit cylinder $(L,\lambda)$ on $(W,GD)$, the following conditions hold:  (i) There is a unique cylinder $(A,\gamma)$ on $(W,D)$ with $(GA,G\gamma) = (L,\lambda)$, and (ii) $(A,\gamma)$ is a colimit cylinder on $(W,D)$.  Note that if $G$ creates $\Phi$-colimits, then $G$ detects, reflects, and conditionally preserves $\Phi$-colimits.  Because of this, we will find that in the study of enriched algebraic theories in the absence of cocompleteness assumptions, conditional preservation plays a more central role than preservation.  All the above terminology can be applied also in the case that $\Phi$ is instead a class of weights, since $\Phi$ determines an associated class of weighted diagrams, namely all those with weights in $\Phi$.
\end{ParSub}

\begin{ParSub}\label{par:vmonads_and_vmonadicity}
Given a monoidal category $\M$, we denote by $\Mon(\M)$ the category of monoids in $\M$.  In particular, given an object $\C$ of a bicategory $\W$, the monoids in the monoidal category $\W(\C,\C)$ are called \textbf{monads on $\C$ in $\W$} and form a category $\Mnd_\W(\C) = \Mon(\W(\C,\C))$.  When $\W$ is the 2-category $\VCAT$ of $\V$-categories, so that $\C$ is a $\V$-category, monads on $\C$ in $\W$ are called $\V$-\textbf{monads}.  As soon as $\V$ has equalizers, we can form the $\V$-category $\C^\TT$ of $\TT$-algebras for each $\V$-monad $\TT$ on $\C$ \cite[Ch. II]{Dub}.  A $\V$-functor $G:\A \rightarrow \C$ is said to be $\V$-\textbf{monadic} (resp. \textbf{strictly} $\V$-\textbf{monadic}) if $G$ has a left adjoint $F$ and the comparison $\V$-functor $\A \rightarrow \C^\TT$ of \cite[Ch. II]{Dub} is an equivalence (resp. isomorphism), where $\TT$ is the $\V$-monad induced by $F \dashv G$.  By the Beck monadicity theorem, formulated in the enriched context by Dubuc in \cite[II.2.1]{Dub}, $G$ is strictly $\V$-monadic if and only if $G$ has a left adjoint and creates conical coequalizers of $G$-contractible pairs, and $G$ is $\V$-monadic if and only if $G$ has a left adjoint and detects, reflects, and conditionally preserves conical coequalizers of $G$-contractible pairs\footnote{Here by a $G$-\textit{contractible pair} we mean a parallel pair $(f,g)$ in $\A$ such that $(Gf,Gg)$ is a contractible pair in the sense of \cite{BaWe}, whereas Dubuc requires also the existence of a coequalizer for $(Gf,Gg)$ and so can write ``preserves'' instead of ``conditionally preserves''.}.  
\end{ParSub}

\begin{ParSub}\label{par:pseudo_slice}
Given an object $K$ of a 2-category $\K$, one can define a 2-category $\K \slash\: K$ called the \textbf{pseudo-slice} 2-category over $K$ in $\K$, as follows.  The objects of $\K \slash\: K$ are 1-cells $p:L \rightarrow K$ in $\K$, written as pairs $(L,p)$, and the 1-cells $(f,\alpha):(L,p) \rightarrow (M,q)$ in $\K \slash\: K$ consist of a 1-cell $f:L \rightarrow M$ in $\K$ and an invertible 2-cell $\alpha:p \Rightarrow qf$ in $\K$.  A 2-cell $\gamma:(f,\alpha) \Rightarrow (g,\beta):(L,p) \rightarrow (M,q)$ is a 2-cell $\gamma:f \Rightarrow g$ in $\K$ such that $q\gamma \cdot \alpha = \beta$.  It is straightforward to show that a 1-cell in $\K \slash\: K$ is an equivalence (resp. isomorphism) as soon as its underlying 1-cell in $\K$ is an equivalence (resp. isomorphism).  In particular, when $\V$ has equalizers, the comparison $\V$-functor \pbref{par:vmonads_and_vmonadicity} associated to a $\V$-monadic $\V$-adjunction $F \dashv G:\A \rightarrow \C$ commutes with the right adjoints \cite[II.1.6]{Dub} and hence determines an equivalence in $\VCAT \slash\: \C$.
\end{ParSub}

\section{Systems of arities}\label{sec:arities}

\begin{DefSub}\label{def:sys_ar}
A \textbf{system of arities} in $\V$ is a fully faithful strong symmetric monoidal $\V$-functor $j:\J \rightarrowtail \uV$.
\end{DefSub}

\begin{RemSub}
In particular, the domain $\J$ of a system of arities is therefore a \textit{symmetric monoidal $\V$-category}, i.e. a $\V$-category $\J$ equipped with an object $I \in \ob\J$, a $\V$-functor $\otimes:\J \otimes \J \rightarrow \J$, and isomorphisms $\ell_J:I \otimes J \rightarrow J$, $r_J:J \otimes I \rightarrow J$, $a_{JKL}:(J \otimes K) \otimes L \rightarrow J \otimes (K \otimes L)$, $s_{JK}:J \otimes K \rightarrow K \otimes J$ that are $\V$-natural in $J,K,L \in \J$ and satisfy the familiar axioms for a symmetric monoidal category.
\end{RemSub}

\begin{ExaSub}[\textbf{Finite cardinals}]\label{exa:fin_card}
Take $\V = \Set$, and let $\J = \FinCard \hookrightarrow \V$ be the full subcategory consisting of the finite cardinals.  Then $\J$ has finite products, given by the usual product of cardinals and the terminal object $1$.  Moreover, the inclusion $\J \hookrightarrow \V$ preserves finite products and so is a strong symmetric monoidal functor. 
\end{ExaSub}

\begin{ExaSub}[\textbf{Finitely presentable objects}]\label{exa:lfp}
Letting $\V$ be \textit{locally finitely presentable as a closed category} \cite{Ke:FL,Pow:EnrLaw}, we can take $\J \hookrightarrow \uV$ to be the full sub-$\V$-category $\Vfp$ consisting of the finitely presentable objects.
\end{ExaSub}

\begin{ExaSub}[\textbf{Unrestricted arities}]\label{exa:unr_ar}
The identity $\V$-functor $\uV \rightarrow \uV$ is a system of arities.
\end{ExaSub}

\begin{ExaSub}[\textbf{Just the unit object $I$}]\label{exa:just_unit}
The one-object full sub-$\V$-category $\{I\} \hookrightarrow \uV$ is a system of arities.  Indeed, observe that $\{I\}$ is isomorphic to the \textit{unit $\V$-category} $\text{I}$, i.e. the $\V$-category with a single object $*$ and $\text{I}(*,*) = I$, with the evident composition and identity arrow.  $\V$-functors $\text{I} \rightarrow \uV$ just correspond to objects in $\V$.  Furthermore, $\text{I}$ carries the structure of a symmetric strict monoidal $\V$-category, trivially, and symmetric monoidal $\V$-functors $\text{I} \rightarrow \uV$ correspond to commutative monoids in $\V$.  Since the unit object $I$ carries the structure of a commutative monoid in $\V$, we obtain a corresponding symmetric monoidal $\V$-functor $j:\text{I} \rightarrow \uV$, which is moreover a strong monoidal $\V$-functor since the monoid structure on $I$ consists of isomorphisms.  Note that the composite $\text{I} \cong \{I\} \hookrightarrow \uV$ equals $j$, so $j$ is fully faithful.  Hence $j$ is a system of arities, and the inclusion $\{I\} \hookrightarrow \uV$ therefore also carries the structure of a system of arities.
\end{ExaSub}

\begin{ExaSub}[\textbf{Finite copowers of the unit object}]\label{exa:fin_copowers_of_I}
Assuming that $\V$ has finite copowers $n\cdot I$ $(n \in \NN)$ of the unit object $I$, the mapping $\NN \rightarrow \ob\V$ sending $n$ to $n \cdot I$ determines an identity-on-homs $\V$-functor
$$j:\NN_{\V} \rightarrowtail \uV$$
where $\NN_{\V}$ is a $\V$-category with $\ob\NN_{\V} = \NN$ and with $\NN_{\V}(n,m) = \uV(n \cdot I,m \cdot I)$.  For all $m, n \in \NN$, the evident canonical morphism $\phi_{mn}:(m \times n)\cdot I \rightarrow (m\cdot I) \otimes (n\cdot I)$ in $\V$ is an isomorphism since the monoidal product $\otimes$ preserves copowers in each variable.  Further, the canonical morphism $\psi:1 \cdot I \rightarrow I$ is also an isomorphism.  Moreover, there is a unique structure of strict symmetric monoidal $\V$-category on $\NN_{\V}$ with unit object $1$ and monoidal product $\times$ such that $j:\NN_{\V} \rightarrow \uV$ is a strong symmetric monoidal $\V$-functor when equipped with $\phi^{-1}$ and $\psi^{-1}$.  The $\V$-category $\NN_{\V}$ is clearly equivalent to the full sub-$\V$-category of $\uV$ consisting of the objects of the form $n \cdot I$ $(n \in \NN)$, employed by Borceux and Day in \cite{BoDay} and there written as $\V_f$.
\end{ExaSub}

\begin{PropSub}\label{thm:sys_ar_equ_subcat}
Any full sub-$\V$-category $\J \hookrightarrow \uV$ containing $I$ and closed under $\otimes$ is a system of arities.  Moreover, assuming the axiom of choice for classes / possibly-large sets, any system of arities $j:\J \rightarrowtail \uV$ is equivalent to a system $j':\J' \hookrightarrow \uV$ of the latter form, in the sense that there is an equivalence of symmetric monoidal $\V$-categories $\J \simeq \J'$ that commutes with $j$ and $j'$.
\end{PropSub}
\begin{proof}
The first claim is immediate.  Regarding the second, we can take $\ob\J' \subseteq \ob\V$ to be either the class of all objects isomorphic to objects in the image of $j$ or, alternatively, the closure of $j(\ob\J) \cup \{I\}$ in $\ob\V$ under $\otimes$.  In either case, the axiom of choice entails that the corestriction $\overline{j}:\J \rightarrow \J'$ of $j$ participates in an adjoint equivalence of $\V$-categories $\overline{j} \dashv k:\J' \rightarrow \J$.  But $\overline{j}$ is a strong symmetric monoidal $\V$-functor, so by \cite[1.5]{Ke:Doctr}, this adjoint equivalence is a symmetric monoidal $\V$-adjunction.
\end{proof}

\begin{RemSub}\label{rem:assume_subcat}
For ease of notation, we shall often write as if given systems of arities have the special form considered in \bref{thm:sys_ar_equ_subcat}.  In general, this is merely a notational convention, in which we harmlessly omit applications of $j$ and instances of the monoidal structure isomorphisms carried by $j$.   However, in view of \bref{thm:sys_ar_equ_subcat}, we will for many purposes even assume, without loss of generality, that a given system is of the indicated special form.  On this basis, the definitions and results in the sequel shall apply equally to systems of arities in the full generality of Definition \bref{def:sys_ar}.  It is surely worthwhile to retain this generality, as the examples of systems of arities given in \bref{exa:fin_card}, \bref{exa:just_unit}, and \bref{exa:fin_copowers_of_I} are not of the special form in question.  \end{RemSub}

\begin{PropSub}\label{thm:sys_ar_dense}
Any system of arities $j:\J \hookrightarrow \uV$ is a dense $\V$-functor, meaning that for each $V \in \uV$, the identity morphism on the $\V$-functor $\YJ(V) = \uV(j-,V):\J^\op \rightarrow \uV$ presents $V$ as a weighted colimit $V = \YJ(V) \star j$.  Equivalently, the evaluation morphisms $\Ev_J:\uV(J,V) \otimes J \rightarrow V$ associated to the objects $J$ of $\J$ present $V$ as a coend $V = \int^{J \in \J}\uV(J,V) \otimes J$ for the $\V$-functor $\uV(j(-),V) \otimes j(-):\J^\op \otimes \J \rightarrow \uV$.
\end{PropSub}
\begin{proof}
By \cite[(5.17)]{Ke:Ba}, the single-object full sub-$\V$-category $\{I\} \hookrightarrow \uV$ is dense, so since $I \in \J$ it follows by \cite[Theorem 5.13]{Ke:Ba} that $j$ is dense.
\end{proof}

\section{Enriched algebraic theories}\label{sec:enr_alg_th,rem:assume_subcat}

Let $j:\J \hookrightarrow \uV$ be a system of arities (\bref{def:sys_ar},\bref{rem:assume_subcat}).  We say that a $\V$-functor is \textit{$\J$-cotensor-preserving} if it preserves all cotensors by objects $J$ of $\J$ (or, rather, their associated objects $j(J)$ of $\V$, \bref{rem:assume_subcat}).

\begin{DefSub}
A $\V$-enriched \textbf{algebraic theory} with arities $\J \hookrightarrow \uV$ (briefly, a $\J$\nolinebreak\mbox{-}\nolinebreak\textbf{theory}) is a $\V$-category $\T$ equipped with a $\J$-cotensor-preserving identity-on-objects $\V$-functor $\tau:\J^\op \rightarrow \T$.
\end{DefSub}

\begin{ExaSub}\label{exa:jtheories}\emptybox
\begin{description}[leftmargin=3.4ex]
\item[\textbf{1. Lawvere theories.}] For the system of arities $\J = \FinCard \hookrightarrow \Set$ of \bref{exa:fin_card}, the resulting notion of $\J$-theory coincides with the notion of algebraic theory defined by Lawvere \cite{Law:PhD}.
\item[\textbf{2. Power's enriched Lawvere theories.}] For the system of arities $\J \hookrightarrow \uV$ of \bref{exa:lfp} in which $\J = \Vfp$ consists of the finitely presentable objects, the resulting notion of $\J$-theory coincides with the notion of \textit{enriched Lawvere theory} defined by Power in \cite{Pow:EnrLaw}.
\item[\textbf{3. Linton's equational theories.}] With $\V = \Set$ and $\J = \Set \hookrightarrow \Set$ the identity functor, the resulting notion of $\J$-theory is the notion of infinitary algebraic theory defined by Linton in \cite{Lin:Eq}.  These were shown by Linton to correspond to arbitrary monads on $\Set$ \cite{Lin:Eq}, and this result follows also from \bref{thm:vth_vmnd} below.
\item[\textbf{4. Dubuc's $\V$-theories; $\V$-monads on $\V$.}]  For the \textit{unrestricted} system of arities with $\J = \uV$ \pbref{exa:unr_ar}, the resulting notion of $\J$-theory is Dubuc's notion of \textit{$\V$-theory} \cite{Dub:EnrStrSem}, which coincides (up to an equivalence) with the notion of $\V$-monad on $\uV$ \pbref{thm:vth_vmnd}.
\item[\textbf{5. Monoids in $\V$. Rings.} $k$-\textbf{algebras.}]  For the system of arities $\{I\} \hookrightarrow \uV$ consisting of just the unit object $I$ of $\V$, $\{I\}$-theories are the same as monoids in the monoidal category $\V$.  Indeed, recall from \bref{exa:just_unit} that this system of arities is isomorphic to the system of arities $j:\text{I}  \rightarrowtail \uV$ given on objects by $* \mapsto I$, so an $\{I\}$-theory is merely an identity-on-objects $\V$-functor $\text{I}^\op = \text{I} \rightarrow \T$, since every $\V$-functor preserves $\{I\}$-cotensors.  Hence an $\{I\}$-theory is merely a one-object $\V$-category $\T$, and these are the same as monoids in $\V$.  When $\V$ is the cartesian closed category of sets $\Set$, where $I = 1$ is the one-point set, $\{1\}$-theories are therefore the same as monoids in the usual sense.  When $\V$ is the category of abelian groups, with unit object $I = \ZZ$ the integers, $\{\ZZ\}$-theories are therefore the same as rings.  Similarly, when $\V$ is the category $\Mod{k}$ of $k$-modules for a commutative ring $k$, $\{k\}$-theories are the same as $k$-algebras.
\item[\textbf{6. Borceux-Day enriched finite power theories.}] For the system of arities $\J = \NN_{\V} \rightarrowtail \uV$ of \bref{exa:fin_copowers_of_I}, the resulting notion of $\J$-theory is essentially the notion of \textit{$\V$-theory} defined by Borceux and Day in \cite{BoDay}, as we shall now demonstrate---except that Borceux and Day restrict attention to a particular kind of closed category $\V$ called a \textit{$\pi$-category}.  By definition, $\NN_{\V}$-cotensors are the same as cotensors by objects of the full sub-$\V$-category $\V_f \hookrightarrow \uV$, recalling that $\V_f$ is equivalent to $\NN_{\V}$ \pbref{exa:fin_copowers_of_I}, and such cotensors are simply (conical) finite powers.  An $\NN_{\V}$-theory is therefore precisely an identity-on-objects $\V$-functor $\tau:\NN_{\V}^\op \rightarrow \T$ that preserves finite powers.  On the other hand, a Borceux-Day $\V$-theory is a \textit{surjective-on-objects} $\V$-functor $\sigma:\V_f^\op \rightarrow \sS$ that preserves finite powers.  Given a Borceux-Day theory $(\sS,\sigma)$, we obtain an associated $\NN_{\V}$-theory $(\T,\tau)$ by factoring the composite $\NN_{\V}^\op \xrightarrow{\sim} \V_f^\op \xrightarrow{\sigma} \sS$ as a composite $\NN_{\V}^\op \xrightarrow{\tau} \T \rightarrow \sS$ consisting of an identity-on-objects $\V$-functor $\tau$ followed by an identity-on-homs $\V$-functor, which is therefore an equivalence of $\V$-categories $\T \simeq \sS$.  In the other direction, given an $\NN_{\V}$-theory $(\T,\tau)$, we obtain\footnote{Note that this passage depends on a choice of pseudo-inverse $\V_f \rightarrow \NN_{\V}$ to the equivalence $\NN_{\V} \xrightarrow{\sim} \V_f$, $n \mapsto n \cdot I$.  Our view is that the theories in question are reasonably seen as conical finite power theories, so that in defining the object $\T(n,1)$ of $n$-ary operations of a particular $\NN_{\V}$-theory $\T$ it is entirely reasonable to make use of the specific finite cardinal $n$.  Contrastingly, the Borceux-Day formulation draws no distinction between the objects of $n$-ary and $m$-ary operations when $n \cdot I$ happens to be \textit{equal} to $m \cdot I$ (on the nose).} a Borceux-Day theory $(\sS,\sigma)$ by similarly factoring the composite $\V_f^\op \xrightarrow{\sim} \NN_{\V}^\op \xrightarrow{\tau} \T$, and again $\sS \simeq \T$.  For cartesian closed $\V$, $\NN_\V$-theories were also studied in \cite[Ch. 6]{Lu:PhD}.
\end{description}
\end{ExaSub}

\begin{ParSub}[\textbf{Forming $\J$-cotensors in a theory $\T$}]\label{par:forming_cot_in_jth}
Observe that each object $J$ of $\J^\op$ is a cotensor $[J,I]$ of the object $I$ of $\J$ by the object $J$ of $\V$.  Indeed the transpose $J \rightarrow \J^\op(J,I) = \uV(I,J)$ of the canonical isomorphism $\ell:I\otimes J \rightarrow J$ is the counit of a representation $\J^\op(-,J) \cong \uV(J,\J^\op(-,I))$.  Hence, given a $\J$-theory $\T$, since the associated $\V$-functor $\tau:\J^\op \rightarrow \T$ preserves $\J$-cotensors, it follows that $J$ is a cotensor $[J,I]$ in $\T$, with counit obtained as the composite
\begin{equation}\label{eq:cot_counit_gammaj}J \rightarrow \J^\op(J,I) \xrightarrow{\tau_{JI}} \T(J,I)\;,\end{equation}
which we denote by $\gamma_J$.

Moreover, every $\J$-theory $\T$ has all $\J$-cotensors:  For each pair of objects $J,K$ of $\J$, the coevaluation morphism
$$\Coev:J \rightarrow \uV(K,J\otimes K) = \J^\op(J\otimes K,K)$$
exhibits $J\otimes K$ as a cotensor $[J,K]$ of $K$ by $J$ in $\J^\op$.  Hence the composite
$$\gamma^K_J = \left(J \xrightarrow{\Coev} \J^\op(J\otimes K,K) \xrightarrow{\tau_{J\otimes K,K}} \T(J\otimes K,K)\right)$$
exhibits $J \otimes K$ as a cotensor $[J,K]$ in $\T$.
\end{ParSub}

\begin{DefSub}\label{def:des_jcot}
An object $C$ of a $\V$-category $\C$ is said to have \textbf{designated} $\J$-\textbf{cotensors} if it is equipped with a specified choice of cotensor $[J,C]$ in $\C$ for each object $J$ of $\J$.  We say that these designated $\J$-cotensors are \textbf{standard}\footnote{We shall find in \bref{thm:jop_algs_bij} and \bref{thm:equiv_def_jth} reasons why this seemingly trifling condition is not only technically relevant but also implicitly inherent in the notion of $\J$-theory.} if $[I,C]$ is just $C$ itself, with the identity morphism $I \rightarrow \C(C,C)$ as counit.  We say that $\C$ has \textbf{designated} $\J$-\textbf{cotensors} if each object of $\C$ has designated $\J$-cotensors.
\end{DefSub}

\begin{ExaSub}\label{exa:std_des_cot_in_jth}
Note that in any $\J$-theory $\T$, the object $I$ has standard designated $\J$-cotensors, namely $[J,I] = J$ \pbref{par:forming_cot_in_jth} for each object $J$ of $\J$.  It shall be convenient to fix a choice of standard designated $\J$-cotensors $[J,K]$ in $\T$ that extends the basic choice $[J,I] = J$ in the case that $K = I$.
\end{ExaSub}

\section{Algebras and morphisms of theories}

\begin{DefSub}
Let $\T$ be a $\J$-theory.
\begin{enumerate}
\item Given a $\V$-category $\C$, a $\T$-\textbf{algebra} in $\C$ is a $\J$-cotensor-preserving $\V$-functor $A:\T \rightarrow \C$.  We shall often call $\T$-algebras in $\uV$ simply $\T$-algebras.
\item We call $\V$-natural transformations between $\T$-algebras $\T$-\textbf{homomorphisms}.  If the object of $\V$-natural transformations $[\T,\C](A,B) = \int_{J \in \T}\C(AJ,BJ)$ exists in $\uV$ for all $\T$-algebras $A,B$ in $\C$, then (in view of \cite[\S 2.2]{Ke:Ba}) $\T$-algebras in $\C$ are the objects of an evident $\V$-category $\Alg{\T}_\C$.  We denote $\Alg{\T}_{\uV}$ by just $\Alg{\T}$.
\end{enumerate}
\end{DefSub}

\begin{RemSub}\label{rem:ca}
Given a $\T$-algebra $A:\T \rightarrow \C$, we call the object $\ca{A} := AI$ of $\C$ the \textbf{carrier} of $A$.  Hence, since each object $J$ of $\T$ is a cotensor $[J,I]$ of $I$ by $J$ in $\T$, the object $AJ$ of $\C$ is a cotensor $[J,\ca{A}]$ of the carrier $\ca{A}$ by $J$ in $\C$.  Since $\T$ has standard designated $\J$-cotensors of $I$ \pbref{exa:std_des_cot_in_jth}, it follows that $A$ equips its carrier $\ca{A}$ with standard designated $\J$-cotensors.
\end{RemSub}

\begin{ExaSub}
\emptybox
\begin{description}[leftmargin=3.4ex]
\item[\textbf{1. General algebras.}]  
For the classes of examples 1, 2, 3, 4, 6 of \bref{exa:jtheories}, i.e. the theories of Lawvere, Power, Linton, Dubuc, and Borceux-Day, we recover in each case the corresponding notion of $\T$-algebra\footnote{For example 6, we obtain just an \textit{equivalence} of the associated $\V$-categories of $\T$-algebras, cf. \bref{exa:jtheories}.}.
\item[\textbf{2. Eilenberg-Moore algebras.}]
Recalling from \bref{exa:jtheories} that $\J$-theories $\T$ for the system of arities with $\J = \uV$ (i.e., Dubuc's $\V$-theories) correspond to $\V$-monads on $\uV$ \pbref{thm:vth_vmnd}, $\T$-algebras in $\uV$ are (up to an equivalence) the same as algebras for the corresponding $\V$-monad \pbref{thm:algs_of_assoc_monad_and_theory}.
\item[\textbf{3.} $R$-\textbf{modules.}]  Recall that an $\{I\}$-theory $\T$ for the system of arities $\{I\} \hookrightarrow \uV$ is merely a one-object $\V$-category $\T$, equivalently, a monoid $R$ in $\V$.  Observe that a $\T$-algebra $M:\T \rightarrow \C$ is merely an arbitrary $\V$-functor, equivalently, an object $M$ of $\C$ equipped with a morphism of monoids $R \rightarrow \C(M,M)$ in $\V$.  When $\C$ is tensored, this is the same as a unital, associative action $R \otimes M \rightarrow M$.  In particular, \textit{$\T$-algebras in $\uV$ are the same as left $R$-modules in $\V$}.  When $\V = \Ab$ is the category of abelian groups, so that $R$ is a ring, $\T$-algebras are therefore just left $R$-modules.  When $\V = \Set$ and $R$ is a group $G$, $\T$-algebras are the same as left $G$-sets.
\end{description}
\end{ExaSub}

\begin{ParSub}
\label{par:forgetful_functor}
Suppose that the category of $\T$-algebras $\Alg{\T}_\C$ in $\C$ exists.  Then \cite[\S 2.2]{Ke:Ba} entails that there is a $\V$-functor
$$\ca{\text{$-$}} = \Ev_I:\Alg{\T}_\C \rightarrow \C$$
given by evaluation at $I$.  Therefore $\ca{\text{$-$}}$ sends each $\T$-algebra $A$ to its carrier $\ca{A}$.  We shall often write $G$ for $\ca{\text{$-$}}$ and $G'$ for the restriction of $G$ to $\Alg{\T}^!_\C$.
\end{ParSub}

Note that $\J^\op$ is a $\J$-theory when equipped with its identity $\V$-functor.  Our study of $\J$-theories will be facilitated by some observations concerning $\J^\op$-algebras.

\begin{LemSub}\label{thm:vfunc_ind_by_des_jcots}
Suppose that a given object $C$ of a $\V$-category $\C$ has standard designated $\J$-cotensors.  Then the induced $\V$-functor $[-,C]:\J^\op \rightarrow \C$ is a $\J^\op$-algebra.  Further, $[-,C]$ strictly preserves the designated $\J$-cotensors $[J,I] = J$ of $I$ in $\J^\op$ \pbref{exa:std_des_cot_in_jth}, i.e., sends them to the designated cotensors $[J,C]$ in $\C$.
\end{LemSub}
\begin{proof}
$[-,C]$ preserves $\J$-cotensors, since for each object $D$ of $\C$, the $\V$-functor $\C(D,[-,C]):\J^\op \rightarrow \uV$ is isomorphic to the composite 
$$\J^\op \xrightarrow{j^\op} \uV^\op \xrightarrow{\uV(-,\T(D,C))} \uV\;,$$
whose factors both preserve $\J$-cotensors.  Next, letting $\delta_J:J \rightarrow \C([J,C],C)$ denote the designated cotensor counit for each $J \in \ob\J$, consider the following diagram
$$
\renewcommand{\objectstyle}{\scriptstyle}
\renewcommand{\labelstyle}{\scriptstyle}
\xymatrix@C=8ex @R=4ex{
J \ar[r]^{\gamma_J} \ar[dd]_{\delta_J} & \J^\op(J,I) \ar@{=}[d] \ar[r]^{[-,C]_{JI}} & \C([J,C],C) \ar[d]^{\C(-,C)_{C,[J,C]}}\\
& \uV(I,J) \ar[dr]_(.3){\uV(I,\delta_J)} & \uV(\C(C,C),\C([J,C],C)) \ar[d]^{\uV(\delta_I,1)}\\
\C([J,C],C) & & \uV(I,\C([J,C],C)) \ar[ll]^\sim
}
$$
in which the bottom row is the canonical isomorphism and $\gamma_J$ is the counit for the cotensor $[J,I] = J$ in $\J^\op$ \pbref{par:forming_cot_in_jth}.  Hence $\gamma_J$ is just the canonical isomorphism $J \xrightarrow{\sim} \uV(I,J)$, so the leftmost cell commutes.  The rightmost cell also commutes, by the $\V$-naturality of $\delta_K$ in $K \in \J$.  But the counit $\delta_I:I \rightarrow \C(C,C)$ for the cotensor $[I,C] = C$ is merely the identity arrow on $C$ in $\C$ \pbref{def:des_jcot}, and it follows that the rightmost exterior face of the diagram is precisely the inverse of bottom face.  Hence the composite of the top face is $\delta_J$, as needed.
\end{proof}

\begin{RemSub}\label{rem:charn_of_jop_alg_assoc_to_obj_w_jcots}
Note that the $\J^\op$-algebra $[-,C]:\J^\op \rightarrow \C$ given in \bref{thm:vfunc_ind_by_des_jcots} may be characterized as the unique $\V$-functor that is given on objects by $J \mapsto [J,C]$ and makes the designated $\J$-cotensor counits $\delta_J:J \rightarrow \C([J,C],C)$ $\V$-natural in $J \in \J$.
\end{RemSub}

\begin{PropSub}\label{thm:jop_algs_bij}
$\J^\op$-algebras $A$ in any given $\V$-category $\C$ are in bijective correspondence with objects $C$ of $\C$ equipped with standard designated $\J$-cotensors.  Under this bijection, $A$ corresponds to its carrier $\ca{A}$ (with the designated $\J$-cotensors furnished by $A$, \bref{rem:ca}) and $C$ corresponds to the $\J^\op$-algebra $[-,C]:\J^\op \rightarrow \C$ of \bref{thm:vfunc_ind_by_des_jcots}.
\end{PropSub}
\begin{proof}
It is straightforward to show that the indicated processes are mutually inverse by applying the second statement in \bref{thm:vfunc_ind_by_des_jcots} and the unique characterization of $[-,C]$ given in \bref{rem:charn_of_jop_alg_assoc_to_obj_w_jcots}.
\end{proof}

\begin{CorSub}\label{thm:equiv_def_jth}
A $\J$-theory is equivalently defined as a $\V$-category $\T$ with $\ob\T = \ob\J$ in which each object $J$ is equipped with the structure of a cotensor $[J,I]$, such that these designated $\J$-cotensors of $I$ are standard \pbref{def:des_jcot}.  The associated $\V$-functor $\tau:\J^\op \rightarrow \T$ is then $[-,I]$ \pbref{rem:charn_of_jop_alg_assoc_to_obj_w_jcots}.
\end{CorSub}
\begin{proof}
By definition, a $\J$-theory $(\T,\tau)$ consists of a $\V$-category $\T$ together with an identity-on-objects $\J^\op$-algebra $\tau:\J^\op \rightarrow \T$, so the result follows from \bref{thm:jop_algs_bij}.
\end{proof}

\begin{PropSub}\label{thm:charns_talgebras}
Given a $\V$-functor $A:\T \rightarrow \C$ on a $\J$-theory $(\T,\tau)$, the following are equivalent:
\begin{enumerate}
\item $A$ is a $\T$-algebra.
\item $A \circ \tau:\J^\op \rightarrow \C$ is a $\J^\op$-algebra.
\item $A$ preserves $\J$-cotensors of $I$.
\end{enumerate}
\end{PropSub}
\begin{proof}
The equivalence of 1 and 2 is immediate once we recall from \bref{par:forming_cot_in_jth} that the $\J$-cotensors in $\T$ are obtained from those in $\J^\op$ via $\tau:\J^\op \rightarrow \T$.  The implication $1 \Rightarrow 3$ is trivial.  In view of the equivalence $1 \Leftrightarrow 2$, the task of proving $3 \Rightarrow 1$ now reduces immediately to the case where $\T = \J^\op$.  Now if $A:\J^\op \rightarrow \C$ preserves $\J$-cotensors of $I$ then, letting $C = AI$, we find that $A$ endows $C$ with standard designated $\J$-cotensors, and by applying the characterization of the induced $\V$-functor $[-,C]:\J^\op \rightarrow \C$ given in \bref{rem:charn_of_jop_alg_assoc_to_obj_w_jcots}, we deduce that $A = [-,C]$, but $[-,C]$ is a $\J^\op$-algebra \pbref{thm:vfunc_ind_by_des_jcots}.
\end{proof}

\begin{DefSub}\label{def:normal_talg}
Given a $\J$-theory $\T$ and a $\V$-category $\C$ with standard designated $\J$-cotensors \pbref{def:des_jcot}, a \textbf{normal} $\T$-\textbf{algebra} (in $\C$) is a $\V$-functor $A:\T \rightarrow \C$ that strictly preserves the designated $\J$-cotensors $[J,I] = J$ of $I$ in $\T$, i.e. sends them to the designated $\J$-cotensors $[J,\ca{A}]$ of $\ca{A} = AI$ in $\C$.  Any normal $\T$-algebra is indeed a $\T$-algebra, by \bref{thm:charns_talgebras}.  If the $\V$-category $\Alg{\T}_\C$ exists, then we denote by $\Alg{\T}^!_\C$ its full sub-$\V$-category consisting of normal $\T$-algebras.
\end{DefSub}

\begin{RemSub}\label{rem:des_cot_v}
Since $\uV$ will typically play the role of the $\V$-category $\C$, we shall endow $\uV$ with standard designated $\J$-cotensors $[J,V]$ of each of its objects $V$ by forcing $[I,V] = V$ and taking $[J,V] = \uV(J,V)$ otherwise.  We write $\Alg{\T}^! = \Alg{\T}_{\uV}^!$.
\end{RemSub}

\begin{PropSub}\label{thm:charns_talgs}
Let $(\T,\tau)$ be a $\J$-theory, and let $A:\T \rightarrow \C$ be a $\V$-functor valued in a $\V$-category with standard designated $\J$-cotensors \pbref{def:des_jcot}.  Then, writing $\ca{A} = AI$, there is a $\V$-natural transformation
\begin{equation}\label{eq:triangle_char_talgs}
\xymatrix@C=8ex @R=4ex{
\T \ar[rr]^A|(.0){}="s1" &                                                     & \C \\
                         &  \J^\op \ar[ul]^\tau \ar[ur]_{[-,|A|]}|{}="t1"  &
 \ar@{}"s1";"t1"|(.6){}="s2"|(.8){}="t2"
 \ar@{=>}"s2";"t2"^{\kappa_A}
}
\end{equation}
\vspace{-3ex}

\noindent such that
\begin{enumerate}
\item $A$ is a $\T$-algebra iff $\kappa_A$ is an isomorphism, iff the triangle commutes up to isomorphism.
\item $A$ is a normal $\T$-algebra iff $\kappa_A$ is an identity, iff the triangle commutes (strictly).
\end{enumerate}
\end{PropSub}
\begin{proof}
Each object $J$ of $\J$ is a cotensor $[J,I]$ in $\T$, so we have a comparison morphism $\kappa_A^J:AJ = A[J,I] \rightarrow [J,AI] = [J,\ca{A}]$ in $\C$, and since $\tau = [-,I]$ \pbref{thm:equiv_def_jth}, these morphisms constitute a $\V$-natural transformation $\kappa_A$ of the needed form.  The first equivalence in 1 is immediate, as is the first equivalence in 2.  Note also that $A$ is a $\T$-algebra iff $A \circ \tau$ is a $\J^\op$-algebra \pbref{thm:charns_talgebras}, and (by essentially the same proof) $A$ is a normal $\T$-algebra iff $A \circ \tau$ is a normal $\J^\op$-algebra.  Hence the remaining equivalences in 1 and 2 follow as soon as we observe that $[-,\ca{A}]$ is a normal $\J^\op$-algebra, by \bref{thm:vfunc_ind_by_des_jcots}.
\end{proof}

\begin{RemSub}\label{rem:vnat_kappa_in_a}
If $\Alg{\T}_\C$ exists and $\C$ has standard designated $\J$-cotensors, then for each fixed object $J \in \ob\T = \ob\J$, the comparison isomorphisms $\kappa_A^J$ associated to $\T$-algebras $A$ constitute a $\V$-natural isomorphism $\kappa^J$ from the evaluation $\V$-functor $\Ev_{[J,I]} = \Ev_J:\Alg{\T}_\C \rightarrow \C$ to the composite $\Alg{\T}_\C \xrightarrow{\Ev_I} \C \xrightarrow{[J,-]} \C$.  Note that for each normal $\T$-algebra $A$, $\kappa^J_A$ is an identity $AJ  = [J,AI]$.
\end{RemSub}

\begin{PropSub}\label{thm:equiv_alg_nalg}
Let $\C$ be a $\V$-category with standard designated $\J$-cotensors, and suppose that the $\V$-category of $\T$-algebras $\Alg{\T}_\C$ exists.  Then $\Alg{\T}_\C$ is equivalent to its full sub-$\V$-category $\Alg{\T}^!_\C$, consisting of all normal $\T$-algebras in $\C$.
\end{PropSub}
\begin{proof}
It suffices to associate to each $\T$-algebra $A:\T \rightarrow \C$ a normal $\T$-algebra $A^!$, which we shall call the \textbf{normalization} of $A$, and an isomorphism $\nu:A \rightarrow A^!$.  By \bref{thm:charns_talgs}, we have a family of isomorphisms
$$\kappa_A^J:AJ \rightarrow [J,\ca{A}]\;\;\;\;(J \in \ob\J = \ob\T)\;,$$
so there is a unique $\V$-functor $A^!$ on $\T$ given on objects by $A^!J = [J,\ca{A}]$ such that the $\kappa_A^J$ constitute a $\V$-natural isomorphism $A \Rightarrow A^!$, and the result follows.
\end{proof}

\begin{DefSub}\label{def:morph_th}
Given $\J$-theories $(\T,\tau)$ and $(\U,\upsilon)$, a \textbf{morphism of} $\J$-\textbf{theories} $A:\T \rightarrow \U$ is a $\V$-functor such that $A \circ \tau = \upsilon$.  We thus obtain a \textbf{category of} $\J$-\textbf{theories}, denoted by $\ThJ$.  Note that $\J^\op$ is an initial object of $\ThJ$.
\end{DefSub}

\begin{RemSub}\label{rem:morph_pres_cot}
By \bref{thm:charns_talgs}, a morphism of $\J$-theories $M:\T \rightarrow \U$ is the same as a normal $\T$-algebra in $\U$ with carrier $I$.  In particular, a morphism of $\J$-theories therefore preserves all $\J$-cotensors.
\end{RemSub}

\begin{ParSub}
Given a morphism of $\J$-theories $M:\T \rightarrow \U$ and a $\V$-category $\C$ for which the $\V$-categories of algebras $\Alg{\T}_\C$ and $\Alg{\U}_\C$ exist, there is a $\V$-functor 
$$M^*:\Alg{\U}_\C \rightarrow \Alg{\T}_\C$$
given on objects by $A \mapsto AM$ and defined in the obvious way on homs.  Since $M$ preserves the unit object $I$, $M^*$ commutes with the `carrier' functors $\ca{\text{$-$}} = \Ev_I$ to $\C$.
\end{ParSub}

\section{\texorpdfstring{$\J$}{J}-stable colimits and pointwise colimits}

An important basic ingredient of our study will be the consideration of certain special colimits of $\T$-algebras in a $\V$-category $\C$, namely the pointwise colimits, and their relation to certain special colimits in $\C$.  Assuming that $\Alg{\T}_\C$ exists, a weighted colimit in $\Alg{\T}_\C$ is said to be a \textit{pointwise colimit} if it is preserved by each of the evaluation $\V$-functors $\Ev_J:\Alg{\T}_\C \rightarrow \C$ $(J \in \ob\J)$, cf. \cite[\S 3.3]{Ke:Ba}.

\begin{DefSub}
Let $\J \hookrightarrow \uV$ be a system of arities.  Given a $\V$-category $\C$ with designated $\J$-cotensors, a weighted colimit in $\C$ is said to be $\J$-\textbf{stable} if it is preserved by each $\V$-functor $[J,-]:\C \rightarrow \C$ $(J \in \ob\J)$.  A weighted diagram $(W,D)$ in $\C$ is $\J$-\textbf{stable} if every colimit $W \star D$ that exists is necessarily $\J$-stable.  Given a $\V$-functor $G:\A \rightarrow \C$, we say that a weighted diagram $(W,D)$ in $\A$ is $G$-\textbf{relatively} $\J$-\textbf{stable} if the weighted diagram $(W,GD)$ is $\J$-stable.
\end{DefSub}

\begin{DefSub}
We say that a weight $W:\B^\op \rightarrow \uV$ is $\J$-\textbf{flat} if all $W$-weighted colimits in $\uV$ are $\J$-stable, i.e. commute with $\J$-cotensors.  By our convention \bref{par:det_cond_pres_refl_colims}, weighted colimits with $\J$-flat weights will be called $\J$-\textit{flat colimits}.  The notion of $\J$-flat weight is an instance of the notion of \textit{$\Phi$-flat} weight \cite{KeSch} in the case that $\Phi$ is the class of weights for $\J$-cotensors\footnote{Putting aside the fact that the cited article restricts attention to weights on small $\V$-categories.}.
\end{DefSub}

\begin{ExaSub}
When $\V$ is cartesian closed, it is well-known that the $\V$-functors $(-)^n:\uV \rightarrow \uV$ preserve (possibly large, conical) filtered colimits\footnote{This can be proved by a very slight variation on the argument given in \cite[3.8]{Ke:FL} for the case in which the filtered colimits in question are assumed small and $\V$ has small filtered colimits.} and (conical) reflexive coequalizers\footnote{This can be proved through an $n$-fold application of \cite[0.17]{Joh:TopTh}.}.  Hence filtered colimits and reflexive coequalizers in $\uV$ are $\NN_{\V}$-stable colimits for the system of arities $\NN_{\V} \rightarrowtail \uV$, $n \mapsto n \cdot 1$ \pbref{exa:fin_copowers_of_I} when the copowers \mbox{$n \cdot 1$} exist.  Weights for filtered colimits are therefore $\NN_{\V}$-flat.  Conical coequalizers of parallel pairs with a specified common section can be described equivalently as conical colimits of shape $\R$, where $\R$ is the category consisting of a single parallel pair of distinct arrows with a common section.  Hence conical $\R$-colimits are $\NN_{\V}$-flat.
\end{ExaSub}

\begin{ExaSub}
If $\V$ is locally finitely presentable as a closed category, then it follows from \cite[4.9]{Ke:FL} that all small (conical) filtered colimits in $\uV$ are $\J$-stable and $\J$-flat for the system of arities $j:\J = \Vfp \hookrightarrow \uV$ consisting of the finitely presentable objects \pbref{exa:lfp}.
\end{ExaSub}

Now let $\T$ be $\J$-theory, let $\C$ be a $\V$-category with standard designated $\J$-cotensors, and assume that the $\V$-category of $\T$-algebras $\Alg{\T}_\C$ exists.  In the following, we employ the notions of detection, reflection, conditional preservation, and creation of $\Phi$-colimits \pbref{par:weighted_limits} for a class $\Phi$ of weighted diagrams.

\begin{PropSub}\label{thm:carrier_vfunc_cr_g_rel_jstb_colims}
\emptybox
\begin{enumerate}
\item The $\V$-functor $G:\Alg{\T}_\C \rightarrow \C$ detects, reflects, and conditionally preserves $G$-relatively $\J$-stable colimits.
\item The restriction $G':\Alg{\T}^!_\C \rightarrow \C$ of $G$ creates $G'$-relatively $\J$-stable colimits.
\end{enumerate}
Moreover, the pointwise colimits in $\Alg{\T}_\C$ (resp. $\Alg{\T}_\C^!$) are precisely those $G$-relatively $\J$-stable colimits $W \star D$ for which $W \star GD$ exists.
\end{PropSub}
\begin{proof}
It suffices to show that 2 holds, for then 1 follows by \bref{thm:equiv_alg_nalg}  and the remaining claim follows by \bref{rem:vnat_kappa_in_a}.  Let us write simply $G$ for $G'$.  Let $(W,D)$ be a $G$-relatively $\J$-stable weighted diagram of shape $\B$ in $\Alg{\T}^!_\C$, and let $W \star GD$ be a colimit of $(W,GD)$ in $\C$, with colimit cylinder $\lambda$.  Then, for each $J \in \ob\J$, this colimit is preserved by the $\V$-functor $[J,-]:\C \rightarrow \C$, so $[J,W \star GD]$ is a colimit 
$$[J,W \star GD] = W \star [J,GD-]$$
with cylinder $[J,-] \circ \lambda$.  But by \bref{rem:vnat_kappa_in_a} we know that for fixed $J$, $[J,GA] = [J,AI] = AJ$ $\V$-naturally in $A \in \Alg{\T}^!$, so in particular $[J,GD-] = [J,(D-)I] = (D-)J:\B \rightarrow \C$ and hence $[J,W \star GD] = [J,W \star (D-)I]$ is a colimit
\begin{equation}\label{eq:col}[J,W \star (D-)I] = W \star (D-)J\;.\end{equation}
Since this is so for every object $J$ of $\T$, there is a unique $\V$-functor $W \star D:\T \rightarrow \C$ given on objects by $J \mapsto W \star (D-)J = [J,W \star (D-)I]$ such that the family consisting of the cylinders $[J,-] \circ \lambda\::\:W \Rightarrow \C((W \star D)J,(D-)J)$ is $\V$-natural in $J \in \T$.  In order to show that $W \star D$ is a normal $\T$-algebra, it suffices to show that the diagram \eqref{eq:triangle_char_talgs} commutes with $A = W \star D$, and to show this we have but to compute that
$$
\begin{array}{ccccccc}
(W \star D)([J,I]) & = & W \star ((D-)[J,I]) & = & W \star [J,(D-)I] & &\\
& = & W \star (D-)J & = & [J,W \star (D-)I] & = & [J,(W \star D)I]
\end{array}
$$
$\V$-naturally in $J \in \J^\op$, using \eqref{eq:col} and the fact that $D$ is a diagram of normal $\T$-algebras.

Hence, $W \star D$ is a colimit of $(W,D)$ in $\Alg{\T}^!_\C$, and it is straightforward to show that the associated cylinder $(W \star D,\lambda')$ is the unique cylinder $(A,\gamma)$ on $(W,D)$ such that $GA = W \star GD$ and $G \circ \gamma = \lambda$.
\end{proof}

\begin{ParSub}\label{par:jflat_and_gabs_colims}
Given a $\V$-functor $G:\A \rightarrow \C$, we say that a weighted diagram $(W,D)$ in $\A$ is $G$-\textbf{absolute} if every colimit $W \star GD$ that exists is absolute.  Note that when $\C$ has designated $\J$-cotensors, $G$-absolute weighted diagrams are necessarily $G$-relatively $\J$-stable.  Observe also that in the case where $\C = \uV$, any weighted diagram $(W,D)$ in $\A$ with a $\J$-flat weight $W$ is necessarily $G$-relatively $\J$-stable.  Hence the preceding Proposition entails the following:
\end{ParSub}

\begin{CorSub}\label{thm:g_cr_jflat_and_gabs_colims}
The $\V$-functor $G:\Alg{\T} \rightarrow \uV$ detects, reflects, and conditionally preserves $\J$-flat colimits and $G$-absolute colimits.  The $\V$-functor $G':\Alg{\T}^! \rightarrow \uV$ creates $\J$-flat colimits and $G$-absolute colimits.
\end{CorSub}

\section{Eleutheric systems of arities and free cocompletion for a class}

In order to enable the construction of free algebras and a correspondence between $\J$-theories and monads in certain bicategories, we will need to impose a certain axiom of `exactness' on our system of arities $\J \hookrightarrow \uV$, as follows.

\begin{DefSub}\label{def:eleutheric}
A system of arities $j:\J \hookrightarrow \uV$ is said to be \textbf{eleutheric} if $\uV$ has weighted colimits for the weights
$$\y_j(V) = \uV(j-,V):\J^\op \rightarrow \uV\;\;\;\;\;\;\;\;\;(V \in \ob\V)$$
and these weights are $\J$-flat.  Let $\PhiJ$ denote the class consisting of all of the above weights $\y_j(V)$.
\end{DefSub}

\begin{ParSub}[\textbf{Basic characterization}]
By definition, a system of arities $j:\J \hookrightarrow \uV$ is eleutheric iff the following conditions hold:
\begin{enumerate}
\item For each object $V$ of $\V$ and each $\V$-functor $T:\J \rightarrow \uV$ there is a weighted colimit
$$\uV(j-,V) \star T = \int^{\JinJ}\uV(J,V)\otimes TJ$$
in $\uV$, and
\item for each object $K$ of $\J$, the canonical morphism
$$\int^{\JinJ} \uV(J,V)\otimes\uV(K,TJ) \longrightarrow \uV\left(K,\int^{\JinJ}\uV(J,V)\otimes TJ\right)$$
is an isomorphism.
\end{enumerate}
\end{ParSub}

\begin{ParSub}[\textbf{Characterization via Kan extensions}]\label{par:charn_ext_via_kan_extns}
Observe that a system of arities $j:\J \hookrightarrow \uV$ is eleutheric iff the following conditions hold:
\begin{enumerate}
\item For every $\V$-functor $T:\J \rightarrow \uV$, the (pointwise) left Kan extension $\Lan_jT:\uV \rightarrow \uV$ of $T$ along $j$ exists, and
\item this left Kan extension is preserved by each $\V$-functor $\uV(K,-):\uV \rightarrow \uV$ with $K \in \ob\J$.
\end{enumerate}
The notion of preservation of Kan extensions is defined in \cite[\S 4.1]{Ke:Ba}.
\end{ParSub}

\begin{ParSub}[\textbf{Characterization via $\Phi$-atomic objects}]\label{par:charn_ext_sys_phiatomic_objs}  A system of arities $j:\J \hookrightarrow \uV$ is eleutheric iff 
\begin{enumerate}
\item $\uV$ is $\PhiJ$-cocomplete, and
\item every object of $\J$ is a $\PhiJ$-atomic object of $\uV$
\end{enumerate}
in the terminology of \cite{KeSch}, putting aside the fact that the latter article defines these terms only for a class of weights with small domains, whereas $\J^\op$ need not be small.  An object $A$ of a $\V$-category $\A$ is said to be $\Phi$-\textbf{atomic} for a class of weights $\Phi$ if $\A(A,-):\A \rightarrow \uV$ preserves $\Phi$-colimits.
\end{ParSub}

\newpage

\begin{ExaSub}\label{exa:ext_sys_ar}
\emptybox
\begin{description}[leftmargin=3.4ex]
\item[\textbf{1. Finitely presentable objects.}] If $\V$ is locally finitely presentable as a closed category, then the system of arities $j:\J = \Vfp \hookrightarrow \uV$ of \bref{exa:lfp} is eleutheric, as we now show.  By \cite[7.2]{Ke:FL}, $\Vfp$ is essentially small and has $\Phi$-colimits for the class $\Phi$ of finite weights, and the inclusion $j$ preserves $\Phi$-colimits.  Hence the weight $\y_j(V) = \uV(j-,V):\J^\op \rightarrow \uV$ preserves $\Phi$-limits, so by \cite[6.12]{Ke:FL}, the $\V$-functor $\y_j(V) \star (-):[\J,\uV] \rightarrow \uV$ preserves $\Phi$-limits since it is a left Kan extension of $\y_j(V)$ along the Yoneda embedding $\J^\op \rightarrow [\J,\uV]$.  In particular, $\y_j(V) \star (-)$ preserves $\Vfp$-cotensors.
\item[\textbf{2. Finite cardinals.}] The system of arities $\FinCard \hookrightarrow \Set$ is eleutheric.  Indeed, as a special case of $1$, the system of arities $\FinSet \hookrightarrow \Set$ is eleutheric, where $\FinSet$ is the category of finite sets, and since the inclusion $\FinCard \hookrightarrow \FinSet$ is an equivalence, the result follows.
\item[\textbf{3. Arbitrary arities.}]  The system of arities with $\J = \uV$ and $j = 1_{\uV}$ is eleutheric, despite the fact that $\J$ may be large.  Indeed, each $T:\uV \rightarrow \uV$ has a left Kan extension along $1_{\uV}$, namely $T$ itself, and this Kan extension is preserved by every $\V$-functor on $\uV$.
\item[\textbf{4. Just the unit object $I$.}] The system of arities $\{I\} \hookrightarrow \uV$ is always eleutheric, as we now show.  Recall that this system is isomorphic to the system $j:\text{I} \rightarrowtail \uV$ with $* \mapsto I$.  Arbitrary $\V$-functors $\text{I} \rightarrow \uV$ may be identified with single objects $V$ of $\V$, and the canonical morphism $r^{-1}_V:V \rightarrow V \otimes I$ exhibits $V \otimes (-):\uV \rightarrow \uV$ as a left Kan extension of $V:\text{I} \rightarrow \uV$ along $j$.  Since $\uV(I,-) \cong 1_{\uV}:\uV \rightarrow \uV$ preserves all weighted limits, the claim is proved.
\item[\textbf{5. Finite copowers of the unit object.}]  The system of arities $j:\NN_{\V} \rightarrowtail \uV$ of \bref{exa:fin_copowers_of_I} is eleutheric for a wide class of closed categories $\V$, as follows.  For an arbitrary symmetric monoidal closed category $\V$, Borceux and Day study in \cite{BoDay:PPKE} those $\V$-categories $\C$ having the following property:  \textit{For any $\V$-functors $P:\A \rightarrow \C$ and $M:\A \rightarrow \B$ such that $\A$ and $\B$ have (conical) finite products and $\A$ is small, if $P$ preserves finite products then the left Kan extension $\Lan_M P$ of $P$ along $M$ exists and preserves finite products.}  In the terminology of the latter article, a cocomplete $\V$-category $\C$ with finite products is said to be \textit{$\pi(\V)$} if $\C$ has this property.  If $\V$ is complete and cocomplete and $\uV$ itself is $\pi(\V)$, as is the case when $\V$ is a $\pi$-\textit{category} in the sense of \cite{BoDay}, then the system of arities $j:\NN_{\V} \rightarrowtail \uV$ is eleutheric, since for each object $V$ of $\V$ we can take $\A = \NN_{\V}^\op$, $P = \y_j(V) = \uV(j-,V)$, $\B = [\NN_{\V},\uV]$, and then since $P$ preserves finite products and $P \star (-):[\NN_{\V},\uV] \rightarrow \uV$ is a left Kan extension of $P$ along the Yoneda embedding $\NN_{\V}^\op \rightarrow [\NN_{\V},\uV]$ it follows that $P \star (-)$ preserves finite products and hence preserves $\NN_{\V}$-cotensors, showing that $j$ is eleutheric.  In particular, if $\V$ is cartesian closed, complete, and cocomplete, then $\uV$ is $\pi(\V)$ by \cite[3.2]{BoDay:PPKE} and hence $j:\NN_{\V} \rightarrowtail \uV$ is eleutheric.  More generally, if $\V$ is cartesian closed and countably cocomplete, then the same methods as in \cite{BoDay:PPKE} allow us to deduce that $P$-weighted colimits commute with finite products in $\uV$ whenever $P:\A = (\A^\op)^\op \rightarrow \uV$ is a product-preserving $\V$-functor on a $\V$-category $\A$ with finite products and a countable set of objects.  Therefore 
\begin{quote}\textit{the system of arities $j:\NN_{\V} \rightarrowtail \uV$ is eleutheric for any countably cocomplete cartesian closed category $\V$.}\end{quote}
\end{description}
\end{ExaSub}

\begin{ParSub}
With reference to \cite[3.7]{KeSch} and \cite[5.35]{Ke:Ba}, we say that a $\V$-functor $F:\A \rightarrow \B$ presents $\B$ as a \textbf{free $\Phi$-cocompletion} of $\A$ if $\B$ is $\Phi$-cocomplete and for every $\Phi$-cocomplete $\V$-category $\C$, composition with $F$ determines an equivalence of categories $\Cocts{\Phi}(\B,\C) \rightarrow \VCAT(\A,\C)$, where $\Cocts{\Phi}(\B,\C)$ denotes the full subcategory of $\VCAT(\B,\C)$ consisting of $\Phi$-cocontinuous $\V$-functors.
\end{ParSub}

\begin{ParSub}\label{par:weights_for_left_kan_extns_along_j}
Observe that a $\V$-category $\C$ is $\PhiJ$-cocomplete if and only if left Kan extensions of $\V$-functors $\J \rightarrow \C$ along $j:\J \hookrightarrow \uV$ exist.  Hence we call the weights $\y_j(V) = \uV(j-,V)$ in $\PhiJ$ the \textit{weights for left Kan extensions along} $j$, and we therefore use the term \textbf{free cocompletion under left Kan extensions along} $j$ for the notion of free $\PhiJ$-cocompletion.
\end{ParSub}

\begin{ThmSub}[\textbf{Characterization via free $\Phi$-cocompletion}]\label{thm:eleutheric}
A system of arities $j:\J \hookrightarrow \uV$ is eleutheric if and only if $j$ presents $\uV$ as a free cocompletion of $\J$ under left Kan extensions along $j$, i.e. a free $\PhiJ$-cocompletion.
\end{ThmSub}

In order to prove this theorem, let us first note that if $\C$ is an arbitrary $\PhiJ$-cocomplete $\V$-category, then in view of \bref{par:weights_for_left_kan_extns_along_j} we have an adjunction
\begin{equation}\label{eq:lanj_adj}\Adjn{\VCAT(\J,\C)}{\Lan_j}{\VCAT(j,\C)}{}{}{\VCAT(\uV,\C)}\end{equation}
whose unit is an isomorphism since $j$ is fully faithful, so that $\Lan_j$ is fully faithful.  The key to proving \bref{thm:eleutheric} now lies in the following result:

\begin{PropSub}\label{thm:jary_charns}
Let $T:\uV \rightarrow \C$ be a $\V$-functor valued in a $\PhiJ$-cocomplete $\V$-category $\C$.  If $j:\J \hookrightarrow \uV$ is eleutheric, then the following are equivalent:
\begin{enumerate}
\item $T$ is $\PhiJ$-cocontinuous.
\item $T$ is a left Kan extension along $j:\J \hookrightarrow \uV$.
\item The component $\xi_T \colon \Lan_j(T \circ j) \Rightarrow T$ at $T$ of the counit $\xi$ of the adjunction \eqref{eq:lanj_adj} is an isomorphism.
\item $T$ preserves the weighted colimits $V = \y_j(V) \star j$ $(V \in \ob\V)$
that exhibit $j$ as a dense $\V$-functor \pbref{thm:sys_ar_dense}.
\end{enumerate}
Without any assumption on $j$, we always have that $2$, $3$, and $4$ are equivalent and are entailed by $1$.
\end{PropSub}
\begin{proof}
The equivalence of 2 and 3 is immediate from the preceding remark.  Statement 4 holds iff the comparison morphism $\y_j(V) \star (T \circ j) \rightarrow T(\y_j(V) \star j) = TV$ is an isomorphism for all $V \in \ob\V$, but this comparison morphism is readily seen to be exactly the component $(\xi_T)_V$ of $\xi_T$ at $V$, so $4 \Leftrightarrow 3$.  Moreover, 1 clearly implies 4.

Now assuming that $j$ is eleutheric, it suffices to show that for any $\V$-functor $P:\J \rightarrow \C$, the left Kan extension $\Lan_j P:\uV \rightarrow \C$ is $\PhiJ$-cocontinuous.  Observe that for each pair of objects $V,W$ of $\V$, the morphisms $\uV(J,-)_{VW}:\uV(V,W) \rightarrow \uV(\uV(J,V),\uV(J,W))$ with $J \in \ob\J$ present $\uV(V,W)$ as an object of $\V$-natural transformations $[\J^\op,\uV](\y_j(V),\y_j(W))$, since $j$ is dense \pbref{thm:sys_ar_dense}.  Hence the $\V$-functors $\y_j(V):\J^\op \rightarrow \uV$ are the objects of a $\V$-category $\PhiJ$ that is equipped with a fully faithful $\V$-functor $\y_j:\uV \rightarrow \PhiJ$ given on objects by $V \mapsto \y_j(V)$.  The idea is that $\uV$ is thus equivalent to the full sub-$\V$-category $\PhiJ$ of the $\V$-functor category $[\J^\op,\uV]$, except that the latter need not exist as a $\V$-category.  Hence since $\C$ is $\PhiJ$-cocomplete, the fixed $\V$-functor $P:\J \rightarrow \C$ determines a $\V$-functor $(-) \star P:\PhiJ \rightarrow \C$, and $\Lan_j P$ factors as the composite
\begin{equation}\label{eq:kan_ext_as_composite}\uV \overset{\y_j}{\longrightarrow} \PhiJ \xrightarrow{(-) \star P} \C\;.\end{equation}
Since $j$ is eleutheric, it follows that $\y_j$ is $\PhiJ$-cocontinuous, sending each $\PhiJ$-colimit $\y_j(V) \star Q$ in $\uV$ to a pointwise colimit $\y_j(V) \star (\y_j \circ Q)$ in the $\V$-category $\PhiJ$.  But by \cite[\S 3.3 (3.23)]{Ke:Ba}, $(-) \star P$ sends pointwise colimits to colimits in $\C$, so the composite \eqref{eq:kan_ext_as_composite} is $\PhiJ$-cocontinuous, i.e. $\Lan_j P$ is $\PhiJ$-cocontinuous as needed.
\end{proof}

We can now finish the proof of Theorem \bref{thm:eleutheric}.  For each $\PhiJ$-cocomplete $\V$-category $\C$, the adjunction \eqref{eq:lanj_adj} restricts to an equivalence between $\VCAT(\J,\C)$ and the full subcategory $\sL$ of $\VCAT(\uV,\C)$ consisting of left Kan extensions along $j$.  If $j$ is eleutheric, then $\sL = \Cocts{\PhiJ}(\uV,\C)$ by \bref{thm:jary_charns}, so $j$ is a free $\PhiJ$-cocompletion.  Conversely, suppose that $j$ is a free $\PhiJ$-cocompletion.  Then, taking $\C = \uV$ in \eqref{eq:lanj_adj}, the right adjoint $\VCAT(j,\uV)$ restricts to an equivalence $\sS \simeq \VCAT(\J,\uV)$ for each of two choices of full replete subcategory $\sS \hookrightarrow \VCAT(\uV,\uV)$, namely (i) $\sS = \sL$ and (ii) $\sS = \Cocts{\PhiJ}(\uV,\uV)$.  But $\Cocts{\Phi}(\uV,\uV) \subseteq \sL$ by \bref{thm:jary_charns} and so it follows that $\Cocts{\PhiJ}(\uV,\uV) = \sL$, i.e. a $\V$-functor $T:\uV \rightarrow \uV$ is $\PhiJ$-cocontinuous if and only if it is a left Kan extension along $j$.  In particular, each $\V$-functor $\uV(J,-):\uV \rightarrow \uV$ with $J \in \ob\J$ is a left Kan extension of $\J(J,-):\J \rightarrow \uV$ along $j$ and so is $\PhiJ$-cocontinuous.  Hence the weights in $\PhiJ$ are $\J$-flat, so $j$ is eleutheric and the theorem is proved.

\section{Free \texorpdfstring{$\T$}{T}-algebras and monadicity}

In the present section, we assume that $\V$ has equalizers.  Let $\T$ be a $\J$-theory for a system of arities $j:\J \hookrightarrow \uV$.

\begin{PropSub}\label{thm:monadicity}
Let $\C$ be a $\V$-category with standard designated $\J$-cotensors, and suppose that $\Alg{\T}_\C$ exists.  Then the following are equivalent: \textnormal{(1)} $G:\Alg{\T}_\C \rightarrow \C$ is $\V$-monadic, \textnormal{(2)} $G':\Alg{\T}^!_\C \rightarrow \C$ is strictly $\V$-monadic, \textnormal{(3)} $G$ has a left adjoint, \textnormal{(4)} $G'$ has a left adjoint.  Further, if $\C = \uV$ and these equivalent conditions hold, then the induced $\V$-monads on $\uV$ conditionally preserve $\J$-flat colimits.
\end{PropSub}
\begin{proof}
The needed equivalence follows immediately from \bref{thm:carrier_vfunc_cr_g_rel_jstb_colims}, \bref{par:jflat_and_gabs_colims}, and the Beck monadicity theorem \pbref{par:vmonads_and_vmonadicity}, since contractible coequalizers are absolute coequalizers.  Now assuming that $\C = \uV$, $G$ conditionally preserves $\J$-flat colimits by \bref{thm:g_cr_jflat_and_gabs_colims}, so if $G$ has a left adjoint $F$ then the induced $\V$-monad $T = GF$ conditionally preserves $\J$-flat colimits as well, since $F$ preserves all colimits.
\end{proof}

We show in \bref{thm:existence_and_vmonadicity_of_talg} below that the equivalent conditions of \bref{thm:monadicity} are satisfied when $\C = \uV$ and $j$ is eleutheric.  Let us assume for the moment that $\Alg{\T}$ exists, though we will soon see that this assumption is unnecessary when $j$ is eleutheric.

\begin{ParSub}\label{par:free_talg_vfunc_on_j}
Observe that for each object $J$ of $\J$, the representable $\V$-functor $\T(J,-):\T \rightarrow \uV$ is a $\T$-algebra, so we obtain a $\V$-functor $\y:\T^\op \rightarrow \Alg{\T}$.  Let $\phi$ denote the composite $\V$-functor
$$\J \xrightarrow{\tau^\op} \T^\op \xrightarrow{\y} \Alg{\T}\;,$$
so that on objects
$$\phi(J) = \T(J,-)\;.$$
\end{ParSub}

\begin{PropSub}\label{thm:free_alg_on_j}
There are isomorphisms
\begin{equation}\label{eq:rel_adj}\Alg{\T}(\phi(J),A) \cong \uV(j(J),GA)\end{equation}
$\V$-natural in $J \in \J$ and $A \in \Alg{\T}$.  Explicitly, the isomorphism \eqref{eq:rel_adj} can be expressed as the composite
\begin{equation}\label{eq:expl_reladj_iso}\Alg{\T}(\phi(J),A) \xrightarrow{G_{\phi(J) A}} \uV(G\phi(J),GA) \xrightarrow{\uV(\gamma_J,1)} \uV(J,GA)\end{equation}
where $\gamma_J:J \rightarrow \T(J,I) = G\phi(J)$ is the cotensor counit.
\end{PropSub}
\begin{proof}
By \bref{thm:charns_talgs} and \bref{rem:vnat_kappa_in_a}, $A(\tau(J)) \cong \uV(J,AI) = \uV(J,GA)$ $\V$-naturally in $A \in \Alg{\T}$, $J \in \J^\op$.  Hence, by Yoneda we compute that
$$\Alg{\T}(\phi(J),A) = \Alg{\T}(\T(\tau(J),-),A) \cong A(\tau(J)) \cong \uV(J,GA)$$
$\V$-naturally in $A \in \Alg{\T}$.  For each object $J$, the unit of the resulting representation is readily seen to be $\gamma_J$, and so we can express the representation isomorphism via \cite[(1.48)]{Ke:Ba} as the composite \eqref{eq:expl_reladj_iso}.
\end{proof}

\begin{LemSub}\label{thm:lem_rel_adj}
Let $G:\A \rightarrow \B$ be a $\V$-functor, and let $J:\C \rightarrow \B$ be a dense $\V$-functor.  Let $E:\C \rightarrow \A$ be a $\V$-functor such that $\A(EC,A) \cong \B(JC,GA)$, $\V$-naturally in $C \in \C$, $A \in \A$.  Then $G$ has a left adjoint if and only if $E$ has a left Kan extension $F$ along $J$, in which case $F$ is left adjoint to $G$.
\end{LemSub}
\begin{proof}
Since $J$ is dense we have $\B(B,GA) \cong [\C^\op,\uV](\B(J-,B),\B(J-,GA)) \cong [\C^\op,\uV](\B(J-,B),\A(E-,A))$, $\V$-naturally in $B \in \B$, $A \in \A$.  Hence if $F$ is a left Kan extension of $E$ along $J$ then $FB = \B(J-,B) \star E$ and so $\A(FB,A) \cong [\C^\op,\uV](\B(J-,B),\A(E-,A)) \cong \B(B,GA)$, $\V$-naturally.  Conversely, if $F$ is left adjoint to $G$ then $\A(FB,A) \cong \B(B,GA) \cong [\C^\op,\uV](\B(J-,B),\A(E-,A))$, so $FB = \B(J-,B) \star E$.
\end{proof}

By \bref{thm:lem_rel_adj}, \bref{thm:free_alg_on_j}, and \bref{thm:monadicity} we deduce the following:

\begin{CorSub}\label{thm:talg_vmonadic_iff_lanjphi_exists}
The $\V$-functor $G:\Alg{\T} \rightarrow \uV$ is $\V$-monadic if and only if the left Kan extension of $\phi:\J \rightarrow \Alg{\T}$ along $j:\J \hookrightarrow \uV$ exists.  If $\Lan_j\phi$ exists, then it is left adjoint to $G$.
\end{CorSub}

Now removing our assumption that $\Alg{\T}$ exists, the following corollary to \bref{thm:talg_vmonadic_iff_lanjphi_exists} is easily obtained on the basis of a theorem of Kelly on the category of models of an enriched sketch:

\begin{CorSub}\label{thm:case_vlb_jsmall}
Suppose that $\V$ is locally bounded \cite[\S 6.1]{Ke:Ba} and that $\J$ is essentially small.  Then $\Alg{\T}$ exists and is complete and cocomplete as a $\V$-category, and the $\V$-functor $G:\Alg{\T} \rightarrow \uV$ is $\V$-monadic.  Further, the induced $\V$-monad preserves small $\J$-flat colimits and conditionally preserves all $\J$-flat colimits.
\end{CorSub}
\begin{proof}
By Theorem 6.11 of \cite{Ke:Ba}, $\Alg{\T}$ is a reflective sub-$\V$-category of the $\V$-functor $\V$-category $[\T,\uV]$ (which exists since $\T$ is essentially small and $\uV$ is cocomplete).  Hence, $\Alg{\T}$ is complete and cocomplete, so since $\J$ is essentially small, the left Kan extension $F = \Lan_j\phi$ exists.  Therefore, by \bref{thm:talg_vmonadic_iff_lanjphi_exists}, $F$ is left adjoint to $G$ and $G$ is $\V$-monadic.  By \bref{thm:monadicity}, the induced $\V$-monad conditionally preserves $\J$-flat colimits and hence preserves small such.
\end{proof}

The preceding Corollary is fairly widely applicable, as Kelly showed that many closed categories $\V$ are locally bounded \cite[\S 6.1]{Ke:Ba}.  Note however that the left adjoint $F$ to $G$ (and hence the induced $\V$-monad) is given just in terms of certain colimits in $\Alg{\T}$, and we have no simple recipe for how these are formed in terms of the basic ingredients $\V$, $\J$, $\T$.

On the other hand, for an \textit{eleutheric} system of arities $j:\J \hookrightarrow \uV$ we shall see that the colimits needed in order to form the left adjoint $F$ are detected and preserved by $G:\Alg{\T} \rightarrow \uV$, and we do have a simple recipe for them.  More generally, even when $j$ is not eleutheric, we shall find that certain theories $\T$ that we call \textit{extensible} admit this same reasoning, and in fact their associated $\V$-categories of algebras always exist, as a consequence.

\begin{DefSub}
We say that a $\V$-functor $T:\J \rightarrow \uV$ is \textbf{extensible} if the left Kan extension of $T$ along $j:\J \hookrightarrow \uV$ exists and is preserved by each $\V$-functor $\uV(J,-):\uV \rightarrow \uV$ with $J \in \ob\J$.  Note that $T$ is extensible iff for each object $V$ of $\V$, the colimit $\uV(j-,V) \star T$ exists and is $\J$-stable.  We say that a $\J$-theory $(\T,\tau)$ is \textbf{extensible} if $\T(\tau-,I):\J \rightarrow \uV$ is extensible.
\end{DefSub}

\begin{RemSub}
Observe that $j:\J \hookrightarrow \uV$ is eleutheric iff every $T:\J \rightarrow \uV$ is extensible.  Hence every $\J$-theory for an eleutheric system of arities is extensible.
\end{RemSub}

\begin{ThmSub}\label{thm:existence_and_vmonadicity_of_talg}
Let $\T$ be an extensible $\J$-theory.  Then 
\begin{enumerate}
\item the $\V$-category of $\T$-algebras $\Alg{\T}$ exists,
\item the $\V$-functor $G = \ca{-}:\Alg{\T} \rightarrow \uV$ is $\V$-monadic,
\item the restriction $G':\Alg{\T}^! \rightarrow \uV$ is strictly $\V$-monadic, and
\item the induced $\V$-monads on $\uV$ conditionally preserve $\J$-flat colimits.
\end{enumerate}
\end{ThmSub}
\begin{proof}
Let $\V'$ be a $\U$-complete enlargement of $\V$ \pbref{par:enl} such that $\J$ is $\U$-small.  The composite inclusion $\J \hookrightarrow \uV \hookrightarrow \underline{\V'}$ is a system of arities, with respect to which $\T$ may be considered as a ($\V'$-enriched) $\J$-theory.  Now $\Alg{\T}_{\uV}$ exists as a $\V'$-category, and again as in \bref{par:free_talg_vfunc_on_j} we obtain a $\V'$-functor $\phi:\J \rightarrow \Alg{\T}_{\uV}$ for which \bref{thm:free_alg_on_j} goes through, \textit{mutatis mutandis}.  Since any $\V'$-monad on $\uV$ is a $\V$-monad, it suffices to show that $G:\Alg{\T}_{\uV} \rightarrow \uV$ is $\V'$-monadic, for then $\Alg{\T}_{\uV}$ is necessarily a $\V$-category and $G$ is $\V$-monadic, so 1 and 2 hold with $\Alg{\T} = \Alg{\T}_{\uV}$, and it follows by \bref{thm:monadicity} that 3 and 4 hold as well.

Again as in \bref{thm:talg_vmonadic_iff_lanjphi_exists} we deduce that $G:\Alg{\T}_{\uV} \rightarrow \uV$ is $\V'$-monadic if and only if the $\V'$-enriched left Kan extension $\Lan_j \phi$ exists.  Letting $V$ be an object of $\V$, it therefore suffices to show that $\uV(j-,V) \star \phi$ exists.  By \bref{thm:carrier_vfunc_cr_g_rel_jstb_colims}, it suffices to show that the colimit $\uV(j-,V) \star (G \circ \phi)$ exists in $\uV$ and is $\J$-stable, but since $G \circ \phi = \T(\tau-,I):\J \rightarrow \uV$ this is immediate from the assumption that $\T$ is extensible.
\end{proof}

\section{Copresheaf-representable profunctors}

For an eleutheric system of arities, we shall show in \S \bref{sec:th_mon_bicat}  that $\J$-theories are the same as monads in a certain bicategory of $\V$-profunctors.  In the present section, we define that bicategory.

\begin{ParSub}\label{par:prof}
Given $\V$-categories $\A$ and $\B$, recall that a $\V$-\textbf{profunctor} $P$ from $\A$ to $\B$, written $P:\A \modto \B$, is a $\V$-functor\footnote{The choice of direction $P:\A \modto \B$ is evidently only one of two possible conventions, but will be convenient when paired with our convention for profunctor composition.} $P:\B^\op \otimes \A \rightarrow \uV$.  These are also called $\V$-\textit{(bi)modules} or $\V$-\textit{distributors}.  Given $\V$-profunctors
\begin{equation}\label{eqn:compbl_profunctors}P:\A \modto \B\;,\;\;\;Q:\B \modto \C\;,\end{equation} 
we say that a $\V$-profunctor $Q \otimes P:\A \modto \C$ is a \textbf{composite} of $P$ and $Q$ if it is a pointwise coend
$$(Q \otimes P)(C,A) = \int^{B \in \B} Q(C,B)\otimes P(B,A)\;,$$
$\V$-naturally in $C \in \C, A \in \A$.  If $\V$ is cocomplete (which we do not assume here), then $\V$-profunctors among \textit{small} $\V$-categories can always be composed and so (with a choice of coends) are the 1-cells of a bicategory $\VProf$ in which the identity 1-cells are the hom profunctors $\A(-,-)$ on small $\V$-categories $\A$.
\end{ParSub}

Recall that every $\V$-functor $F:\A \rightarrow \B$ determines $\V$-profunctors $F_* = \B(-,F-):\A \modto \B$ and $F^* = \B(F-,-):\B \modto \A$.

\begin{DefSub}
Let $j:\J \hookrightarrow \uV$ be a system of arities.  A $\V$-profunctor $M:\J \modto \J$ is said to be \textbf{copresheaf-representable} if it is a composite $j^* \otimes S_*$ for some $\V$-functor $S:\J \rightarrow \uV$.  Observe that $M$ is copresheaf-representable iff
\begin{equation}\label{eq:def_copr_rep}M \cong \uV(j-,S-)\;:\;\J^\op \otimes \J \rightarrow \uV\end{equation}
for some $\V$-functor $S:\J \rightarrow \uV$, since $\uV(j-,S-)$ is always a composite $j^* \otimes S_*$.
\end{DefSub}

\newpage

\begin{ExaSub}
\emptybox
\begin{description}[leftmargin=3.4ex]
\item[\textbf{1. Representable endo-profunctors on $\uV$.}] For the unrestricted system of arities with $\J = \uV$, a $\V$-profunctor $M:\uV \modto \uV$ is copresheaf-representable if and only if it is \textit{representable} in the usual sense, i.e., iff $M = T_* = \uV(-,T-)$ for some $\V$-functor $T:\uV \rightarrow \uV$. 
\item[\textbf{2. Objects of $\V$.}] The system of arities $\{I\} \hookrightarrow \uV$ is isomorphic to the system $\text{I} \rightarrowtail \uV$ with $* \mapsto I$, and $\V$-profunctors $\text{I} \modto \text{I}$ are just single objects of $\V$.  All of them are copresheaf-representable.
\end{description}
\end{ExaSub}

\begin{PropSub}\label{thm:charn_copr_rep}
A $\V$-profunctor $M:\J \modto \J$ is copresheaf-representable if and only if $M(-,-):\J^\op \otimes \J \rightarrow \uV$ preserves $\J$-cotensors in its first variable.  Writing
$$M_I := M(I,-):\J \longrightarrow \uV\;,$$
there is a canonical $\V$-natural transformation
\begin{equation}\label{eq:zeta}\zeta^M\;:\;M \Longrightarrow \uV(j-,M_I-)\end{equation}
such that $\zeta^M$ is an isomorphism if and only if $M$ is copresheaf-representable.
\end{PropSub}
\begin{proof}
$M$ preserves $\J$-cotensors in its first variable iff $M(-,K):\J^\op \rightarrow \uV$ is a $\J^\op$-algebra for each object $K$ of $\J$, but by \bref{thm:charns_talgs} this is so iff the comparison transformations
$$M(-,K) \Longrightarrow \uV(j-,M(I,K))\;\;\;\;\;\;(K \in \J)$$
are isomorphisms.  The latter constitute a $\V$-natural transformation $\zeta^M$ of the needed form \eqref{eq:zeta}.  If $M$ is copresheaf-representable, with $M \cong \uV(j-,S-)$, then each $M(-,K) \cong \uV(j-,SK):\J^\op \rightarrow \uV$ is a $\J^\op$-algebra by \bref{thm:vfunc_ind_by_des_jcots}, so $\zeta^M$ is iso.  The converse implication is immediate.
\end{proof}

\begin{PropSub}\label{thm:comp_copr_rep}
Let $M,N:\J \modto \J$ be copresheaf-representable $\V$-profunctors for an eleutheric system of arities $j:\J \hookrightarrow \uV$.  Then there exists a composite $M \otimes N:\J \modto \J$, and this composite is copresheaf-representable.
\end{PropSub}
\begin{proof}
Letting $M^\sharp_I := \Lan_j M_I:\uV \rightarrow \uV$ denote the left Kan extension of $M_I = M(I,-):\J \rightarrow \uV$ along $j$, we claim that the copresheaf-representable $\V$-profunctor $\uV(j-,M^\sharp_I N_I-)$ is a composite $M \otimes N$.  Indeed, we can employ the assumption that $j$ is eleutheric to compute that 
$$
\begin{array}{lll}
\uV(j(J),M^\sharp_I N_I L) & =     & \uV(J,\int^{K \in \J}\uV(K,N_I L)\otimes M_I K)  \\
                           & \cong & \int^{K \in \J}\uV(K,N_I L) \otimes \uV(J,M_I K) \\
                           & \cong & \int^{K \in \J}M(J,K) \otimes N(K,L)
\end{array}
$$
$\V$-naturally in $J,L \in \J$, and that, in particular, the coends on the second and third lines exist.
\end{proof}

\begin{CorSub}\label{thm:bic_copr_rep_prof}
Given an eleutheric system of arities $j:\J \hookrightarrow \uV$, the copresheaf-representable $\V$-profunctors $M:\J \modto \J$ are the 1-cells of a bicategory with just one object, namely $\J$.
\end{CorSub}
\begin{proof}
Since the identity $\V$-profunctor $\J(-,-) = \uV(j-,j-)$ is co\-pre\-sheaf-re\-pre\-sent\-able, this follows from \bref{thm:comp_copr_rep}.
\end{proof}

\begin{DefSub}\label{def:crprofj}
We denote by $\CRProfJ$ the one-object \textbf{bicategory of copresheaf-representable} $\V$-\textbf{profunctors} of \bref{thm:bic_copr_rep_prof}.  Note that $\CRProfJ$ can equally be viewed as a monoidal category whose objects are copresheaf-representable $\V$-profunctors.
\end{DefSub}

\section{Theories as monads in a bicategory of profunctors}\label{sec:th_mon_bicat}

Letting $j:\J \hookrightarrow \uV$ be an eleutheric system of arities, we show herein that $\J$-theories are the same as monads in the bicategory of copresheaf-representable $\V$-profunctors on $\J$.  The idea of describing algebraic theories as certain profunctor monads was pursued in \cite[Ch. V]{JohWr} for internal algebraic theories in a topos, and a closely related description of the enriched theories of Borceux and Day is given in \cite[2.6.1]{BoDay}.  In the present context, we shall require some basic theory on $\V$-profunctor monads, as follows.

\begin{ParSub}
To begin, let us recall that commutative $K$-algebras for a commutative ring $K$ can be defined as commutative monoids $A$ in the symmetric monoidal category of $K$-modules or, equivalently, as commutative rings $A$ equipped with a ring homomorphism $K \rightarrow A$.  It is well-known that this fact has a non-commutative analogue for a given monoid $K$ in an arbitrary monoidal category $\C$ with reflexive coequalizers that are preserved by $\otimes$ in each variable.  In this context, the category $\Bimod_\C(K)$ of $K$-bimodules is a monoidal category when we define the monoidal product $M \otimes_K N$ of a pair of $K$-bimodules $M,N$ as the coequalizer of the reflexive pair $\alpha \otimes N, M \otimes \beta:M \otimes K \otimes N \rightarrow M \otimes N$ where $\alpha$ and $\beta$ are the right and left actions carried by $M$ and $N$, respectively.  The general result is then as follows and is straightforward to prove; cf. \cite[3.7]{Ch:DistLLTh}.
\end{ParSub}

\begin{PropSub}\label{thm:two_descns_of_algs_over_a_monoid}
Given a monoid $K$ in a monoidal category $\C$ with reflexive coequalizers that are preserved by $\otimes$ in each variable, there is an isomorphism 
$$\Mon(\Bimod_\C(K)) \;\cong\; K \slash \:\Mon(\C)$$
between the category $\Mon(\Bimod_\C(K))$ of monoids in $\Bimod_\C(K)$ and the coslice category under $K$ in the category $\Mon(\C)$ of monoids in $\C$.
\end{PropSub}

\begin{ExaSub}
If $\V$ is cocomplete, then given any small set $S$ we can take $\C = \VProf(S,S)$ to be the monoidal category of all $\V$-profunctors $S \modto S$ on the discrete $\V$-category $S$.  Writing $\VCAT(S)$ for the category whose objects are $\V$-categories $\K$ with $\ob\K = S$, with identity-on-objects $\V$-functors as morphisms, it is well-known that $\Mon(\C) \cong \VCAT(S)$.  It is equally well-known that the monoidal category $\Bimod_\C(\K)$ of bimodules for a monoid $\K$ in $\C$ is isomorphic to the monoidal category $\VProf(\K,\K)$ of $\V$-profunctors $\K \modto \K$.  Therefore \bref{thm:two_descns_of_algs_over_a_monoid} entails that monads on $\K$ in $\VProf$ are equivalently described as $\V$-categories $\A$ with an identity-on-objects $\V$-functor $\K \rightarrow \A$, as captured by the following $\V$-enriched variant of a result of Justesen \cite[p. 202]{JohWr}:
\end{ExaSub}

\begin{CorSub}\label{thm:vprof_mnds}
Suppose that $\V$ is cocomplete, and let $\K$ be a small $\V$-category.  Then we have an isomorphism
$$\Mnd_{\VProf}(\K) \;\cong\; \K \slash \:\VCAT(\ob\K)$$
between the category of monads on $\K$ in $\VProf$ and the coslice category under $\K$ in $\VCAT(\ob\K)$. Given an object $(\A,E:\K \rightarrow \A)$ of the given coslice category, the associated $\V$-profunctor is $\A(E-,E-):\K \modto \K$.
\end{CorSub}
\begin{proof}
$\K \slash \:\VCAT(\ob\K) \cong \K \slash \:\Mon(\VProf(\ob\K,\ob\K)) \cong \Mon(\VProf(\K,\K))$ $=$ $\Mnd_{\VProf}(\K)$.
\end{proof}

\begin{ThmSub}\label{thm:th_as_mnd_crprofj}
For an eleutheric system of arities $j:\J \hookrightarrow \uV$, there is an isomorphism
$$\ThJ \;\cong\; \Mnd_{\CRProfJ}(\J)$$
between the category $\ThJ$ of $\J$-theories and the category $\Mnd_{\CRProfJ}(\J)$ of monads on $\J$ in the bicategory $\CRProfJ$ of copresheaf-representable $\V$-profunctors.
\end{ThmSub}
\begin{proof}
Let $\V'$ be a $\U$-cocomplete enlargement of $\V$ \pbref{par:enl} such that $\J$ is $\U$-small, and write $\B = \eProf{\V'}$ for the bicategory of $\V'$-profunctors among $\U$-small $\V'$-categories.  Then since $\V \hookrightarrow \V'$ sends $\U$-small $\V$-enriched coends to $\V'$-enriched coends, the bicategory $\C = \CRProfJ$ is (w.l.o.g.) a locally full sub-bicategory of $\B$.  Therefore $\Mnd_\C(\J)$ is a full subcategory of $\Mnd_\B(\J)$.  But by \bref{thm:vprof_mnds}, we know that $\Mnd_\B(\J) \cong \X$ where $\X = \J \slash \:\eCAT{\V'}(\ob\J)$, and so this isomorphism restricts to a full embedding $\Mnd_\C(\J) \rightarrowtail \X$.  On the other hand we also have a full embedding $\ThJ \rightarrowtail \X$ given by $(\T,\tau) \mapsto (\T^\op,\tau^\op)$.  These two embeddings are both injective on objects, so it suffices to show that they have the same image.  Given an object $(\A,E)$ of $\X$, it suffices to show that the associated $\V'$-profunctor $\A(E-,E-):\J \modto \J$ is a copresheaf-representable $\V$-profunctor if and only if $(\A^\op,E^\op)$ is a $\J$-theory.  We may assume that $\A$ is a $\V$-category.  Now $(\A^\op,E^\op)$ is a $\J$-theory iff $E^\op:\J^\op \rightarrow \A^\op$ preserves $\J$-cotensors, but since the $\V$-functors
$$\A(-,K) = \A^\op(K,-):\A^\op \rightarrow \uV\;\;\;\;\;K \in \ob\A = \ob\J$$
preserve and jointly reflect cotensors, this is equivalent to the assertion that
each of the $\V$-functors
$$\A(E-,EK) = \A(E-,K):\J^\op \rightarrow \uV\;\;\;\;\;K \in \ob\A = \ob\J$$
preserves $\J$-cotensors, equivalently, that the $\V$-profunctor $\A(E-,E-)$ is copresheaf-representable.
\end{proof}

\section{Equivalence between \texorpdfstring{$\J$}{J}-theories and \texorpdfstring{$\J$}{J}-ary monads}

Let $j:\J \hookrightarrow \uV$ be an eleutheric system of arities.

\begin{DefSub}
A $\J$-\textbf{ary} $\V$-\textbf{functor} is a $\V$-functor that preserves left Kan extensions along $j:\J \hookrightarrow \uV$.
\end{DefSub}

\begin{RemSub}\label{rem:jary_phi_cocts}
Note that $\J$-ary $\V$-functors are the same as $\PhiJ$-cocontinuous $\V$-functors \pbref{def:eleutheric}.  Let us denote by $\Cocts{\PhiJ}$ the locally full sub-2-category of $\VCAT$ with 1-cells all $\J$-ary $\V$-functors between $\PhiJ$-cocomplete $\V$-categories.
\end{RemSub}

\begin{ExaSub}\label{exa:jary_endofunctors}
\emptybox
\begin{description}[leftmargin=3.4ex]
\item[1. Finitary endofunctors.]  Letting $\V$ be locally finitely presentable as a closed category, the $\J$-ary $\V$-functors $T:\uV \rightarrow \uV$ for the system of arities $\J = \V_{fp} \hookrightarrow \uV$ are exactly those $\V$-functors that are \textit{finitary} in the sense employed in \cite{Ke:FL}.  Indeed, by \cite[7.6]{Ke:FL}, $T$ is finitary if and only if $T$ is a left Kan extension along $j:\Vfp \hookrightarrow \uV$, so \bref{thm:jary_charns} yields the needed equivalence.  In particular, in the classical case where $\V = \Set$ and $\V_{fp} = \FinSet$, the $\FinSet$-ary endofunctors on $\Set$ are exactly the finitary endofunctors in the usual sense.  Hence we find that the $\FinCard$-ary endofunctors on $\Set$ are precisely the finitary endofunctors as well.
\item[2. Unrestricted arities, arbitrary endofunctors.] For the system of arities $j = 1_{\uV}:\J = \uV \rightarrow \uV$, every $\V$-functor $T:\uV \rightarrow \uV$ is a $\uV$-ary $\V$-functor.  Indeed, given any $S:\J = \uV \rightarrow \uV$, the identity morphism $1_S:S \Rightarrow Sj$ exhibits $S$ as a left Kan extension of $S$ along $j$, and this left Kan extension is clearly preserved by $T$.
\end{description}
\end{ExaSub}

Our next objective is to show that the one-object 2-category of $\J$-ary endo-$\V$-functors on $\uV$ is equivalent to the one-object bicategory $\CRProfJ$ of copresheaf-representable $\V$-profunctors on $\J$, equivalently, that $\Cocts{\PhiJ}(\uV,\uV) \simeq \CRProfJ$ as monoidal categories.  To this end, we continue with the following:

\begin{LemSub}
There is an adjunction
\begin{equation}\label{eq:adj_prof_func}\Adjn{\VPROF(\J,\J)}{\Lambda}{\Theta}{}{}{\VCAT(\J,\uV)}\end{equation}
that restricts to an equivalence of categories
$$\CRProf_\J(\J,\J) \;\;\;\;\simeq\;\;\;\; \VCAT(\J,\uV)\;.$$
Here, $\VPROF(\J,\J)$ denotes the ordinary category of $\V$-profunctors $\J \modto \J$.
\end{LemSub}
\begin{proof}
On objects, we define $\Lambda M = M_I = M(I,-)$ and $\Theta S = \uV(j-,S-)$.  These assignments extend to functors in the evident way.  By \bref{thm:charn_copr_rep}, we have a canonical morphism
$$\zeta^M\colon M \Longrightarrow \uV(j-,M_I-) = \Theta \Lambda M$$
for each object $M$ of $\VPROF(\J,\J)$, and these constitute a natural transformation $\zeta:1 \Rightarrow \Theta\Lambda$.  Also, we have a canonical isomorphism
$$\xi^S\colon\Lambda\Theta S = \uV(I,S-) \Longrightarrow S$$
for each object $S$ of $\VCAT(\J,\uV)$, and these consistute a natural transformation $\xi:\Lambda\Theta \Rightarrow 1$.  It is straightforward to verify the triangular equations in order to show that $\Lambda \nsststile{\xi}{\zeta} \Theta$.  Since the counit $\xi$ is an isomorphism, $\Theta$ is fully faithful and the adjunction restricts to an equivalence between $\VCAT(\J,\V)$ and the full subcategory of $\VPROF(\J,\J)$ consisting of all $M$ for which $\zeta^M$ is an isomorphism, but these are exactly the copresheaf-representable $\V$-profunctors \pbref{thm:charn_copr_rep}.
\end{proof}

\begin{CorSub}\label{thm:comp_equiv_cats}
There are equivalences of categories
$$\CRProf_\J(\J,\J) \;\;\;\simeq\;\;\;\VCAT(\J,\uV)\;\;\;\simeq\;\;\;\Cocts{\PhiJ}(\uV,\uV)\;,$$
where the rightmost equivalence is obtained via \bref{thm:eleutheric} as a restriction of the left Kan-extension adjunction \eqref{eq:lanj_adj} in the case that $\C = \uV$.
\end{CorSub}

\begin{ThmSub}\label{thm:crprofj_bieq_jary_end}
There is a (bi)equivalence of bicategories
$$\CRProfJ \;\;\;\;\simeq\;\;\;\; \Cocts{\PhiJ}(\uV)$$
between the one-object bicategory $\CRProfJ$ of copresheaf-representable $\V$-profunctors on $\J$ and the one-object 2-category $\Cocts{\PhiJ}(\uV)$ of $\J$-ary endo-$\V$-functors on $\uV$.  Equivalently, $\CRProfJ \simeq \Cocts{\PhiJ}(\uV)$ as monoidal categories.
\end{ThmSub}
\begin{proof}
By \bref{thm:comp_equiv_cats}, there is an adjoint equivalence $\Omega \dashv \Gamma:\CRProfJ \rightarrow \Cocts{\PhiJ}(\uV)$ between the ordinary categories underlying the monoidal categories in question.  It suffices to show that $\Omega$ carries the structure of a strong monoidal functor, for then it follows by \cite[1.5]{Ke:Doctr} that $\Omega \dashv \Gamma$ underlies an adjunction in the 2-category of monoidal categories, but since the unit and counit of the resulting adjunction are isomorphisms, it is an equivalence.

The functor $\Omega:\Cocts{\PhiJ}(\uV) \rightarrow \CRProfJ$ is given by
$$\Omega(T) = \uV(j-,Tj-)\;:\;\J^\op \otimes \J \rightarrow \uV$$
naturally in $T \in \Cocts{\PhiJ}(\uV)$.  Given objects $S,T$ of $\Cocts{\PhiJ}(\uV)$, we have isomorphisms
$$
\begin{array}{lll}
(\Omega(T) \otimes \Omega(S))(J,L) & = & \int^{K \in \J}\uV(J,TK) \otimes \uV(K,SL)\\
                                   & \cong & \uV(J,\int^{K \in \J} TK \otimes \uV(K,SL))\\
                                   & \cong & \uV(J,TSL)\\
                                   & = & \Omega(T \circ S)(J,L)
\end{array}
$$
$\V$-natural in $J \in \J^\op$, $L \in \J$.  Indeed, the first of the two indicated isomorphisms results from the assumption that $j$ is eleutheric, and the second obtains since $T \cong \Lan_j(T \circ j)$.  Hence we have a composite isomorphism
$$m^{TS}:\Omega(T) \otimes \Omega(S) \overset{\sim}{\Longrightarrow} \Omega(T \circ S),$$
which is evidently natural in $T,S \in \Cocts{\PhiJ}(\uV)$.  Noting that $\Omega(1_{\uV}) = \uV(j-,j-) = \J(-,-)$, let us denote by
$$e:\J(-,-) \overset{\sim}{\Longrightarrow} \Omega(1_{\uV})$$
the identity morphism.  We claim that $(\Omega,e,m)$ is a monoidal functor.  Using the definition of $m^{TS}$, we find that each of its components
$$m^{TS}_{JL}:\int^{\inJ{K}}\uV(J,TK) \otimes \uV(K,SL) \longrightarrow \uV(J,TSL)\;\;\;\;\;\;(J,L \in \J)$$
is induced by the $\V$-natural family consisting of the composites
$$m^{TS}_{JKL} = \left(\uV(J,TK) \otimes \uV(K,SL) \xrightarrow{1 \otimes T_{K,SL}} \uV(J,TK) \otimes \uV(TK,TSL) \xrightarrow{c} \uV(J,TSL)\right)$$
$(K \in \J)$ where $c$ is the composition morphism.  Hence, using the definition of the monoidal product $\otimes$ in $\CRProfJ$, it follows that the needed diagrammatic associativity law \cite[II.1.2 MF3]{EiKe} for the monoidal functor $(\Omega,e,m)$ amounts to the commutativity of the diagram
$$
\xymatrix{
\uV(J,UK) \otimes \uV(K,TL) \otimes \uV(L,SM) \ar[d]_{1 \otimes m^{TS}_{KLM}} \ar[rr]^(.55){m^{UT}_{JKL} \otimes 1} & & \uV(J,UTL) \otimes \uV(L,SM) \ar[d]^{m^{UT,S}_{JLM}}\\
\uV(J,UK) \otimes \uV(K,TSM) \ar[rr]_{m^{U,TS}_{JKM}} & & \uV(J,UTSM)
}
$$
for all $S,T,U \in \Cocts{\PhiJ}(\uV)$ and all $J,K,L,M \in \J$.  This commutativity is straightforwardly verified through a single diagrammatic computation.  Similarly, the left unit law \cite[II.1.2 MF1]{EiKe} for $(\Omega,e,m)$ amounts to the statement that
$$m^{1_{\uV},\:T}_{JKL}:\uV(J,K) \otimes \uV(K,TL) \longrightarrow \uV(J,TL)$$
is merely the composition morphism, for all $T \in \Cocts{\PhiJ}(\uV)$ and all $J,K,L \in \J$, and this is immediate from the definition.  The right unit law \cite[II.1.2 MF2]{EiKe} amounts to the statement that
$$m^{T,\:1_{\uV}}_{JKL}:\uV(J,TK) \otimes \uV(K,L) \longrightarrow \uV(J,TL)$$
is equal to the composite
$$\uV(J,TK) \otimes \uV(K,L) \xrightarrow{1 \otimes T_{KL}} \uV(J,TK) \otimes \uV(TK,TL) \overset{c}{\longrightarrow} \uV(J,TL)$$
for all $T,J,K,L$, and this is also immediate from the definition.
\end{proof}

\begin{DefSub}
A $\J$-\textbf{ary $\V$-monad} on $\uV$ is a $\V$-monad $\TT$ on $\uV$ whose underlying $\V$-functor $T:\uV \rightarrow \uV$ is a $\J$-ary $\V$-functor.  Equivalently, a $\J$-ary $\V$-monad on $\uV$ is a monad on $\uV$ in the 2-category $\Cocts{\PhiJ}$.  Therefore, $\J$-ary $\V$-monads on $\uV$ form a category
$$\MndJ(\uV) = \Mnd_{\Cocts{\PhiJ}}(\uV),$$
the category of monads on $\uV$ in $\Cocts{\PhiJ}$.  We show in \bref{thm:charn_jary_mnd_via_jflat_colims} that a $\V$-monad $\TT$ on $\uV$ is $\J$-ary if and only if $T$ conditionally preserves $\J$-flat colimits.
\end{DefSub}

\begin{ThmSub}\label{thm:jth_jary_mnd}
There is an equivalence
$$\ThJ \;\;\;\;\simeq\;\;\;\; \MndJ(\uV)$$
between the category $\ThJ$ of $\J$-theories and the category $\MndJ(\uV)$ of $\J$-ary $\V$-monads on $\uV$.
\end{ThmSub}
\begin{proof}
By \bref{thm:th_as_mnd_crprofj}, $\ThJ$ is isomorphic to the category $\Mnd_{\CRProfJ}(\J)$ of monads on $\J$ in the one-object bicategory $\CRProfJ$, and by \bref{thm:crprofj_bieq_jary_end} we have a (bi)equivalence of bicategories $\CRProfJ \simeq \Cocts{\PhiJ}(\uV)$, so
$$\ThJ \cong \Mnd_{\CRProfJ}(\J) \simeq \Mnd_{\Cocts{\PhiJ}(\uV)}(\uV) = \MndJ(\uV)\;.$$
\end{proof}

\begin{DefSub}
Let us denote by
$$\mathsf{m}:\ThJ \rightarrow \MndJ(\uV) \;\;\;\;\;\;\text{and}\;\;\;\;\;\; \mathsf{t}:\MndJ(\uV) \rightarrow \ThJ$$
the equivalences obtained in \bref{thm:jth_jary_mnd}.
\end{DefSub}

As a corollary, theories with unrestricted arities in $\V$ are equivalent to $\V$-monads on $\uV$, as Dubuc showed for complete and well-powered $\V$ \cite{Dub:EnrStrSem}: 

\begin{CorSub}\label{thm:vth_vmnd}
There is an equivalence
$$\Th_{\uV} \;\;\;\;\simeq\;\;\;\; \Mnd_{\VCAT}(\uV)$$
between the category $\Th_{\uV}$ of $\uV$-theories (i.e., $\J$-theories for $\J = \uV$) and the category $\Mnd_{\VCAT}(\uV)$ of arbitrary $\V$-monads on $\uV$.
\end{CorSub}
\begin{proof}
By \bref{exa:jary_endofunctors}, every $\V$-monad $\TT$ on $\uV$ is $\uV$-ary, so this follows from Theorem \bref{thm:jth_jary_mnd}.
\end{proof}

\begin{PropSub}\label{thm:th_assoc_to_jary_mnd}
Let $\TT = (T,\eta,\mu)$ be a $\J$-ary $\V$-monad on $\uV$, and let $\J_\TT$ denote the full sub-$\V$-category of the Kleisli $\V$-category $\uV_\TT$ whose objects are exactly those of $\J$.  Then 
$$\mathsf{t}(\TT) = \J_\TT^\op\;,$$
i.e., the $\J$-theory $\mathsf{t}(\TT)$ associated to $\TT$ via the equivalence \bref{thm:jth_jary_mnd} is precisely $\J_\TT^\op$.
\end{PropSub}
\begin{proof}
In the notation of \bref{par:vmonads_and_vmonadicity}, $\mathsf{t}:\MndJ(\uV) \rightarrow \ThJ$ is the composite
$$\MndJ(\uV) = \Mon(\Cocts{\PhiJ}(\uV)) \xrightarrow{\Mon(\Omega)} \Mon(\CRProfJ) \xrightarrow{\sim} \Th_\J$$
where $\Mon(\Omega)$ is induced by the monoidal functor $\Omega:\Cocts{\PhiJ}(\uV) \rightarrow \CRProfJ$ defined in the proof of \bref{thm:crprofj_bieq_jary_end}.  The functor $\Mon(\Omega)$ sends $\TT$ to a monoid $(\Omega(T),\mathsf{e},\mathsf{m})$ in $\CRProfJ$ with
$$\Omega(T) = \uV(j-,(T \circ j)-):\J \modto \J$$
$$\mathsf{e} = \left(\J(-,-) \xrightarrow{e} \Omega(1_{\uV}) \xrightarrow{\Omega(\eta)} \Omega(T)\right)$$
$$\mathsf{m} = \left(\Omega(T) \otimes \Omega(T) \xrightarrow{m^{TT}} \Omega(T \circ T) \xrightarrow{\Omega(\mu)} \Omega(T)\right)$$
where $e$ and $m$ are the monoidal structure morphisms carried by $\Omega$.  Using the definition of $e$ and $m$, we find that 
$$\mathsf{e} = \left(\uV(j-,j-) \xrightarrow{\uV(j-,(\eta \circ j)-)} \uV(j-,(T \circ j)-)\right)$$
and that the components of $\mathsf{m}$ are the composite morphisms
$$\mathsf{m}_{JL} = \left(\int^{\inJ{K}}\uV(J,TK) \otimes \uV(K,TL) \xrightarrow{m^{TT}_{JL}} \uV(J,TTL) \xrightarrow{\uV(J,\mu_L)} \uV(J,TL)\right)$$
with $J,L \in \ob\J$, induced by the composites
\begin{equation}\label{eq:kl_compn_morphs}
\renewcommand{\objectstyle}{\scriptstyle}
\renewcommand{\labelstyle}{\scriptstyle}
\xymatrix{
\uV(J,TK) \otimes \uV(K,TL) \ar[r]^(.46){1 \otimes T_{K,TL}} & \uV(J,TK) \otimes \uV(TK,TTL) \ar[r]^(.6)c & \uV(J,TTL) \ar[r]^(0.54){\uV(J,\mu_L)} & \uV(J,TL)
}
\end{equation}
$(J,K,L \in \J)$, where $c$ is the composition morphism.  Now $\T := \mathsf{t}(\TT)$ is the $\J$-theory corresponding to $(\Omega(T),\mathsf{e},\mathsf{m})$ under the isomorphism $\Mon(\CRProfJ) \cong \Th_\J$ of \bref{thm:th_as_mnd_crprofj}.  Hence, by the definition of the latter isomorphism, $\T = \A^\op$ where $\A$ is a $\V$-category with $\ob\A = \ob\J$ and
$$\A(J,K) = \Omega(T)(J,K) = \uV(J,TK) = \J_\TT(J,K)\;\;\;\;\;\;(J,K \in \ob\J).$$
The composition morphisms carried by $\A$ are exactly the morphisms \eqref{eq:kl_compn_morphs} inducing $\mathsf{m}$, and the identity arrow $1_J^\A \in \A_0(J,J) = \V(J,TJ)$ on $J$ in $\A$ corresponds (under Yoneda) to $\mathsf{e}_{-J} = \uV(j-,\eta_J):\J(-,J) = \uV(j-,J) \Rightarrow \uV(j-,TJ)$, i.e. $1_J^\A = \eta_J$.  Hence $\A = \J_\TT$, so  $\T = \J_\TT^\op$.
\end{proof}

\begin{DefSub}
Given a $\J$-ary $\V$-monad $\TT$ on $\uV$, we call the $\J$-theory $\mathsf{t}(\TT) = \J_\TT^\op$ the \textbf{Kleisli} $\J$-\textbf{theory} for $\TT$.
\end{DefSub}

Let us now assume that $\V$ has equalizers.

\begin{LemSub}\label{thm:ind_monad_iso_assoc_monad}
Let $(\T,\tau)$ be a $\J$-theory, let $F \dashv G:\Alg{\T} \rightarrow \uV$ denote the associated monadic $\V$-adjunction \pbref{thm:existence_and_vmonadicity_of_talg}, and let $\TT$ denote the $\V$-monad induced by this $\V$-adjunction.  Then $\TT$ is a $\J$-ary $\V$-monad, and its Kleisli $\J$-theory $\J_\TT^\op$ is isomorphic to $\T$, i.e. $\mathsf{t}(\TT) \cong\ \T$.  Consequently $\TT \cong \mathsf{m}(\T)$.
\end{LemSub}
\begin{proof}
It suffices to show that $\mathsf{t}(\TT) \cong\ \T$, for then $\mathsf{m}(\T) \cong \mathsf{m}(\mathsf{t}(\TT)) \cong \TT$.  We have a diagram
$$
\xymatrix{
\Alg{\T} \ar[rr]^C_\sim & & \uV^\TT\\
& \uV \ar[ul]^F \ar[ur]^{F^\TT} \ar[r]_{F_\TT} & \uV_{\TT} \ar@{ >->}[u]_E\\
\T^\op \ar@{ >->}[uu]^{\y} & \J \ar@{^{(}->}[u]^j \ar[l]^{\tau^\op} \ar[r]_{F'_\TT} & \J_\TT \ar@{^{(}->}[u]_{H}
}
$$
where $C$ is the comparison $\V$-functor \cite[II.1]{Dub} (an equivalence), $F^\TT$ is the Eilenberg-Moore left adjoint, $\y$ is the Yoneda $\V$-functor, $H$ is the inclusion, $F_\TT$ is the Kleisli left adjoint, $F'_\TT$ is the restriction of $F_\TT$, and $E$ is the comparison $\V$-functor for the Kleisli $\V$-adjunction (and hence is fully faithful).  The two triangular cells commute by \cite[II.1.6]{Dub}, and the small square clearly commutes.  The two composites in the remaining cell are isomorphic, since by \bref{thm:talg_vmonadic_iff_lanjphi_exists}, $F \circ j = (\Lan_j \phi) \circ j \cong \phi = \y \circ \tau^\op$.  Therefore for each object $J \in \ob\T^\op = \ob\J$ we have an isomorphism
$$C\y\tau^\op(J) \cong EHF'_\TT(J)\;,$$
and since $\tau^\op$ and $F'_\TT$ are identity-on-objects, this is an isomorphism
\begin{equation}\label{eq:isos_on_objs}C\y(J) \cong EH(J)\;.\end{equation}
Hence since $EH$ is fully faithful, there is a unique identity-on-objects $\V$-functor $Q:\T^\op \rightarrow \J_\TT$ such that the isomorphisms \eqref{eq:isos_on_objs} constitute a $\V$-natural isomorphism $C \circ \y \cong EHQ$.  Since $C \circ \y$ and $EH$ are both fully faithful, it follows that $Q$ is fully faithful and hence is an isomorphism. 
\end{proof}

\begin{ThmSub}\label{thm:algs_of_assoc_monad_and_theory}
\emptybox
\begin{enumerate}
\item Given a $\J$-theory $\T$,
$$\uV^{\mathsf{m}(\T)} \;\;\cong\;\; \Alg{\T}^! \;\;\simeq \;\;\Alg{\T}\;,$$
i.e., the $\V$-category of Eilenberg-Moore algebras for the associated $\V$-monad $\mathsf{m}(\T)$ is isomorphic to the $\V$-category of normal $\T$-algebras and equivalent to the $\V$-category of $\T$-algebras.
\item Given a $\J$-ary $\V$-monad $\TT$ on $\uV$,
$$\uV^\TT \;\;\cong\;\; \Alg{\mathsf{t}(\TT)}^! \;\;\simeq\;\; \Alg{\mathsf{t}(\TT)}\;,$$
i.e., the $\V$-category of Eilenberg-Moore $\TT$-algebras is isomorphic to the $\V$-category of normal $\mathsf{t}(\TT)$-algebras and equivalent to the $\V$-category of $\mathsf{t}(\TT)$-algebras.
\end{enumerate}
Moreover, the above are equivalences in the pseudo-slice $\VCAT \slash \uV$ \pbref{par:pseudo_slice} when the above $\V$-categories are equipped with the evident $\V$-functors to $\uV$.
\end{ThmSub}
\begin{proof}
1.  By \bref{thm:existence_and_vmonadicity_of_talg}, we have $\V$-monadic $\V$-adjunctions $F \dashv G:\Alg{\T} \rightarrow \uV$ and $F' \dashv G':\Alg{\T}^! \rightarrow \uV$, and the latter is strictly $\V$-monadic.  Since $G'$ is the restriction of $G$ along the equivalence $\Alg{\T}^! \hookrightarrow \Alg{\T}$, it follows that the respective $\V$-monads $\TT$ and $\TT'$ induced by these $\V$-adjunctions are isomorphic, and by \bref{thm:ind_monad_iso_assoc_monad} we have moreover that $\TT' \cong \TT \cong \mathsf{m}(\T)$.  Therefore,
$$\uV^{\mathsf{m}(\T)} \cong \uV^{\TT'} \cong \Alg{\T}^! \simeq \Alg{\T}\;.$$

2. For an arbitrary $\J$-ary $\V$-monad $\TT$ on $\uV$, we know that $\TT \cong \mathsf{m}(\mathsf{t}(\TT))$, so by 1 we deduce that
$$\uV^{\TT} \cong \uV^{\mathsf{m}(\mathsf{t}(\TT))} \cong \Alg{\mathsf{t}(\TT)}^! \simeq \Alg{\mathsf{t}(\TT)}\;.$$
\end{proof}

\section{Characterization theorem for \texorpdfstring{$\J$}{J}-algebraic categories over \texorpdfstring{$\V$}{V}}

Let $j:\J \hookrightarrow \uV$ be an eleutheric system of arities, and assume that $\V$ has equalizers.  Let us consider both $\Alg{\T}$ (for a $\J$-theory $\T$) and $\uV^{\TT}$ (for a $\V$-monad $\TT$ on $\uV$) as objects of the pseudo-slice 2-category $\VCAT \slash\: \uV$ \pbref{par:pseudo_slice} via the `forgetful' $\V$-functors.

\begin{DefSub}
Let $G:\A \rightarrow \uV$ be a $\V$-functor, exhibiting $\A$ as an object of $\VCAT \slash \uV$.  $G$ is $\J$-\textbf{algebraic} if there is an equivalence $\A \simeq \Alg{\T}$ in $\VCAT \slash \uV$ for some $\J$-theory $\T$.  $G$ is $\J$-\textbf{monadic} if there is an equivalence $\A \simeq \uV^{\TT}$ in $\VCAT \slash \uV$ for some $\J$-ary $\V$-monad $\TT$ on $\uV$.
\end{DefSub}

\begin{ThmSub}\label{thm:charn_jalg_cats_over_v}
For a $\V$-functor $G:\A \rightarrow \uV$, the following are equivalent:
\begin{enumerate}
\item $G$ is $\J$-algebraic.
\item $G$ is $\J$-monadic.
\item $G$ has a left adjoint, and $G$ detects, reflects, and conditionally preserves $G$-relatively $\J$-stable colimits. 
\item $G$ has a left adjoint, and $G$ detects, reflects, and conditionally preserves $\J$-flat colimits and $G$-absolute colimits.
\item $G$ is $\V$-monadic and the induced $\V$-monad conditionally preserves $\J$-flat colimits.
\end{enumerate}
\end{ThmSub}
\begin{proof}
The equivalence of 1 and 2 follows from \bref{thm:algs_of_assoc_monad_and_theory}, and the implication $1 \Rightarrow 3$ follows from \bref{thm:existence_and_vmonadicity_of_talg} and \bref{thm:carrier_vfunc_cr_g_rel_jstb_colims}.  Also, 3 implies 4, by \bref{par:jflat_and_gabs_colims}.  If 4 holds, then the Beck monadicity theorem \pbref{par:vmonads_and_vmonadicity} entails that $G$ is $\V$-monadic, and since the left adjoint $F$ preserves colimits and $G$ conditionally preserves $\J$-flat colimits it follows that the induced $\V$-monad $T = GF$ conditionally preserves $\J$-flat colimits.  Lastly, if 5 holds, then since the weights in $\PhiJ$ are $\J$-flat and $\uV$ is $\PhiJ$-cocomplete it follows that the induced $\V$-monad preserves $\PhiJ$-colimits, and hence 2 holds.
\end{proof}

\begin{CorSub}\label{thm:charn_jary_mnd_via_jflat_colims}
Let $\TT = (T,\eta,\mu)$ be a $\V$-monad on $\uV$.  Then $\TT$ is a $\J$-ary $\V$-monad if and only if $T$ conditionally preserves $\J$-flat colimits.
\end{CorSub}
\begin{proof}
Invoke \bref{thm:charn_jalg_cats_over_v} with respect to $G^\TT:\uV^\TT \rightarrow \uV$.
\end{proof}

\section{Appendix: Finite copowers of \texorpdfstring{$I$}{I} versus \texorpdfstring{$\Phi$}{Phi}-presentable objects}\label{sec:appendix}

We show herein that the finite copowers of $I$ in $\V$ are not in general the same as the $\Phi$-presentable objects of $\uV$ for the class $\Phi$ of weights for (conical) finite products (cf. \S \bref{sec:intro}), even when $\V$ is a $\pi$-category in the sense of \cite{BoDay}.

Let us first recall some terminology and facts from \cite{LaRo}, wherein $\V$ is assumed complete and cocomplete.  In (only) the present section, we follow \cite{LaRo} in taking the term \textit{weight} to mean any $\V$-functor $\B^\op \rightarrow \uV$ that is \textit{small} in the sense of \cite{DayLack}.  Given a class of weights $\Phi$, an object $V$ of $\uV$ is said to be \textit{$\Phi$-presentable} in the terminology of \cite{LaRo} if $\uV(V,-):\uV \rightarrow \uV$ preserves $\Phi^+$-colimits, where $\Phi^+$ is the class of $\Phi$-flat weights, i.e. all those weights $W$ for which $W$-colimits commute with $\Phi$-limits in $\uV$.  

Let $\V_\Phi \hookrightarrow \uV$ denote the full sub-$\V$-category consisting of the $\Phi$-presentable objects, and let $\Phi^{+-}$ denote\footnote{We follow \cite{KeSch} in using this notation, except that that article restricts attention to weights on small $\V$-categories.} the class of all weights $U$ such that $U$-limits commute with $\Phi^+$-colimits in $\uV$.

\begin{PropSub}
$\V_\Phi$ is closed in $\uV$ under the taking of (i) $\Phi^{+-}$-colimits, (ii) $\Phi$-colimits, and (iii) retracts.
\end{PropSub}
\begin{proof}
(i).  Given a colimit $U \star D$ in $\uV$ where $U:\C^\op \rightarrow \uV$ lies in $\Phi^{+-}$ and $D:\C \rightarrow \uV$ is valued in $\V_\Phi$, we have that $\uV(U \star D,V) \cong \{U,\uV(D-,V)\}$, $\V$-naturally in $V \in \uV$, but $\uV(DC,-)$ preserves $\Phi^+$-colimits for each object $C$ of $\C$, so since $U$-limits commute with $\Phi^+$-colimits in $\uV$, it follows that $\uV(U \star D,-)$ preserves $\Phi^+$-colimits, i.e. $U \star D$ is $\Phi$-presentable.  (ii). $\Phi \subseteq \Phi^{+-}$, so this follows from (i).  (iii). Retracts in $\V$ can be described equivalently as \textit{idempotent-splittings} in $\uV$, which are conical colimits of diagrams of a particular shape $\F$ (namely a one-object category $\F$ with a single non-trivial idempotent \cite[5.8]{Ke:Ba}), and are equivalently described as conical \textit{limits} of the same shape $\F$ \cite[Vol. 1, \S 6.5]{Bor}.  Letting $U$ denote the weight for conical limits and colimits of shape $\F = \F^\op$, it suffices by (i) to show that $U \in \Phi^{+-}$.  Given an arbitrary weight $W:\B^\op \rightarrow \uV$, since $W$ is small and $\uV$ is cocomplete, it follows that $\uV$ has all $W$-colimits, so we obtain an ordinary functor $W \star (-):\VCAT(\B,\uV) \rightarrow \V$.  But idempotent-splittings are preserved by any functor and hence by $W \star (-)$, and it follows that $U$-limits commute with $W$-colimits in $\uV$.  In particular, $U \in \Phi^{+-}$ as needed.
\end{proof}

\begin{ParSub}
Now choose any commutative ring $R$ for which there exists a finitely generated non-free projective $R$-module $M$, and let $\V$ be the category of $R$-modules, which is a $\pi$-category in the sense of Borceux and Day \cite{BoDay}.  Here $I = R$.  Letting $\Phi$ be the class of weights for finite products, the closure $\Phi(R)$ of $\{R\}$ in $\uV$ under $\Phi$-colimits consists of exactly the finite copowers of $R$, i.e. the finitely generated free $R$-modules.  Since the full sub-$\V$-category $\V_\Phi \hookrightarrow \uV$ consisting of all $\Phi$-presentable objects contains $R$ and is closed under $\Phi$-colimits, $\Phi(R) \subseteq \V_\Phi$.  But $M$ is a retract of a finitely generated free $R$-module $N$, so since $N \in \V_\Phi$ and $\V_\Phi$ is closed under retracts, $M \in \V_\Phi$ as well, yet $M \notin \Phi(R)$ as $M$ is not free.
\end{ParSub}

\bibliographystyle{amsplain}
\bibliography{bib}

\end{document}

%% file: pre.tex
\usepackage{amsmath, amsthm, amssymb, stmaryrd}
\usepackage[all]{xy}
\usepackage{fullpage}
\usepackage{titlesec}
\usepackage{mathrsfs}
\usepackage{amsbsy}
\usepackage{turnstile}
\usepackage{bbm}
\usepackage{yhmath}
\usepackage{tensor}
\usepackage{graphics}
\usepackage{enumitem}
\usepackage{calligra}
\usepackage{verbatim}

\usepackage{setspace}

\usepackage[pdftex,bookmarks=true]{hyperref}
\usepackage[svgnames]{xcolor}
\hypersetup{
    linkbordercolor = {PaleTurquoise},
    citebordercolor = {PaleGreen},
}

\makeatletter
\newcommand*{\@old@slash}{}\let\@old@slash\slash
\def\slash{\relax\ifmmode\delimiter"502F30E\mathopen{}\else\@old@slash\fi}
\makeatother

\titleformat{\section}{\normalsize\bfseries}{\thesection}{1em}{}
\titleformat{\subsection}{\normalsize\bfseries}{\thesubsection}{1em}{}

\numberwithin{equation}{subsection}

\theoremstyle{plain}

\newtheorem{PropSub}[subsection]{Proposition}
\newtheorem{LemSub}[subsection]{Lemma}
\newtheorem{CorSub}[subsection]{Corollary}
\newtheorem{ThmSub}[subsection]{Theorem}

\theoremstyle{definition}

\newtheorem{DefSub}[subsection]{Definition}
\newtheorem{ExaSub}[subsection]{Example}
\newtheorem{RemSub}[subsection]{Remark}
\newtheorem{ParSub}[subsection]{}

\newcommand*{\emptybox}{\leavevmode\hbox{}}

\newdir{ >}{{}*!/-10pt/@{>}}

\DeclareMathAlphabet{\mathpzc}{OT1}{pzc}{m}{it}
\DeclareMathAlphabet{\mathcalligra}{T1}{calligra}{m}{n}

\newcommand{\bref}[1]{\textnormal{\ref{#1}}}
\newcommand{\pbref}[1]{\textnormal{(\ref{#1})}}

\newcommand{\A}{\ensuremath{\mathscr{A}}}
\newcommand{\B}{\ensuremath{\mathscr{B}}}
\newcommand{\C}{\ensuremath{\mathscr{C}}}

\newcommand{\F}{\ensuremath{\mathscr{F}}}

\newcommand{\normalJ}{\ensuremath{\mathscr{J}}}
\newcommand{\J}{\ensuremath{{\kern -0.4ex \mathscr{J}}}}

\newcommand{\K}{\ensuremath{\mathscr{K}}}
\newcommand{\sL}{\ensuremath{\mathscr{L}}}
\newcommand{\M}{\ensuremath{\mathscr{M}}}

\newcommand{\R}{\ensuremath{\mathscr{R}}}
\newcommand{\sS}{\ensuremath{\mathscr{S}}}
\newcommand{\T}{\ensuremath{\mathscr{T}}}
\newcommand{\U}{\ensuremath{\mathscr{U}}}
\newcommand{\V}{\ensuremath{\mathscr{V}}}
\newcommand{\W}{\ensuremath{\mathscr{W}}}
\newcommand{\X}{\ensuremath{\mathscr{X}}}

\newcommand{\uV}{\ensuremath{\mkern2mu\underline{\mkern-2mu\mathscr{V}\mkern-6mu}\mkern5mu}}

\newcommand{\eCAT}[1]{\ensuremath{#1\textnormal{-\text{CAT}}}}
\newcommand{\VCAT}{\ensuremath{\V\textnormal{-\text{CAT}}}}

\newcommand{\VProf}{\ensuremath{\V\textnormal{-\text{Prof}}}}
\newcommand{\eProf}[1]{\ensuremath{#1\textnormal{-\text{Prof}}}}
\newcommand{\VPROF}{\ensuremath{\V\textnormal{-\text{PROF}}}}
\newcommand{\Cocts}[1]{\ensuremath{#1\textnormal{-\text{Cocts}}}}

\newcommand{\NN}{\ensuremath{\mathbb{N}}}

\newcommand{\TT}{\ensuremath{\mathbb{T}}}

\newcommand{\ZZ}{\ensuremath{\mathbb{Z}}}

\newcommand{\ob}{\ensuremath{\operatorname{\textnormal{\textsf{ob}}}}}

\newcommand{\Lan}{\ensuremath{\operatorname{\textnormal{\textsf{Lan}}}}}

\newcommand{\ca}[1]{\scalebox{0.85}{\raisebox{0.3ex}{$|$}} #1\scalebox{0.85}{\raisebox{0.3ex}{$|$}}}

\newcommand{\Th}{\ensuremath{\textnormal{Th}}}

\newcommand{\ThJ}{\ensuremath{\Th_{\kern -0.5ex \normalJ}}}
\newcommand{\SubThJ}{\ensuremath{\textnormal{SubTh}_{\kern -0.5ex \normalJ}}}
\newcommand{\Vfp}{\ensuremath{\V_{\kern -0.5ex fp}}}

\newcommand{\Ev}{\ensuremath{\textnormal{\textsf{Ev}}}}
\newcommand{\Coev}{\ensuremath{\textnormal{\textsf{Coev}}}}

\newcommand{\y}{\ensuremath{\textnormal{\textsf{y}}}}

\newcommand{\YJ}{\ensuremath{\textnormal{\textsf{y}}_j}}
\newcommand{\inJ}[1]{#1 \in \normalJ\kern -1.2ex}
\newcommand{\JinJ}{\inJ{J}}
\newcommand{\VCATJ}{\VCAT_{\kern -0.7ex \normalJ}}
\newcommand{\PhiJ}{\Phi_{\kern -0.8ex \normalJ}}
\newcommand{\MndJ}{\Mnd_{\kern -0.8ex \normalJ}}

\newcommand{\Set}{\ensuremath{\operatorname{\textnormal{\text{Set}}}}}
\newcommand{\FinCard}{\ensuremath{\operatorname{\textnormal{\text{FinCard}}}}}
\newcommand{\FinSet}{\ensuremath{\operatorname{\textnormal{\text{FinSet}}}}}

\newcommand{\Ab}{\ensuremath{\operatorname{\textnormal{\text{Ab}}}}}

\newcommand{\Mod}[1]{\ensuremath{#1\textnormal{-\text{Mod}}}}

\newcommand{\Bimod}{\ensuremath{\textnormal{\text{Bimod}}}}

\newcommand{\Mnd}{\ensuremath{\operatorname{\textnormal{\text{Mnd}}}}}

\newcommand{\Mon}{\ensuremath{\operatorname{\textnormal{\text{Mon}}}}}

\newcommand{\Alg}[1]{\ensuremath{#1\kern -.5ex\operatorname{\textnormal{-\text{Alg}}}}}

\newcommand{\CRProf}{\ensuremath{\operatorname{\textnormal{\text{CRProf}}}}}
\newcommand{\CRProfJ}{\ensuremath{\CRProf_{\kern -0.6ex \normalJ}}}

\newcommand{\Adjn}[6]{\xymatrix {#1 \ar@/_0.5pc/[rr]_{#2}^(0.4){#4}^(0.6){#5}^{\top} & & #6 \ar@/_0.5pc/[ll]_{#3}}}
\newcommand{\Equiv}[6]{\xymatrix {#1 \ar@/_0.5pc/[rr]_{#2}^(0.4){#4}^(0.6){#5}^{\sim} & & #6 \ar@/_0.5pc/[ll]_{#3}}}

\def\modto{\mathop{\rlap{\hspace{1.1ex}$\circ$}{\longrightarrow}}\nolimits}

\newcommand{\op}{\ensuremath{\textnormal{op}}}

\newcommand{\pushoutcorner}{\ar@{}[dr]|(.3)\ulcorner}
\newcommand{\pullbackcorner}{\ar@{}[dr]|(.3)\lrcorner}

\newcommand{\cmt}[1]{}